\documentclass{amsart}

\usepackage[english]{babel}

\usepackage[letterpaper,top=2cm,bottom=2cm,left=3cm,right=3cm,marginparwidth=1.75cm]{geometry}

\usepackage{amssymb}
\usepackage{amsmath}
\usepackage{mathabx}

\usepackage{scalerel,stackengine}
\stackMath
\newcommand\restr[2]{{
  \left.\kern-\nulldelimiterspace 
  #1 
  \littletaller 
  \right|_{#2} 
  }}

\newcommand{\littletaller}{\mathchoice{\vphantom{\big|}}{}{}{}}
\newcommand\reallywidehat[1]{%
\savestack{\tmpbox}{\stretchto{%
  \scaleto{%
    \scalerel*[\widthof{\ensuremath{#1}}]{\kern-.6pt\bigwedge\kern-.6pt}%
    {\rule[-\textheight/2]{1ex}{\textheight}}
  }{\textheight}%
}{0.5ex}}%
\stackon[1pt]{#1}{\tmpbox}%
}
\usepackage{amsthm}
\usepackage{xfrac}
\usepackage{bbm}
\usepackage{mathrsfs}
\usepackage{indentfirst}
\usepackage[colorlinks=true, linkcolor=blue, urlcolor=blue]{hyperref}
\usepackage{cases}
\usepackage{graphicx}

\title{A Blowup Solution of Multispeed Klein-Gordon System in Space Dimension Two with Small Initial Data}
\author{Xilu Zhu}
\date{}

\address{Dept. of Mathematics, University of Southern California,
 Los Angeles, CA 90089}
\email{xiluzhu@usc.edu}

\begin{document}
\maketitle
\newtheorem{lemma}{Lemma}[section]
\newtheorem{cor}[lemma]{Corollary}
\newtheorem{prop}[lemma]{Proposition}
\newtheorem{defn}[lemma]{Definition}
\newtheorem{theorem}[lemma]{Theorem}

\theoremstyle{definition}
\newtheorem{remark}[lemma]{Remark}

\begin{abstract}
We find an example to illustrate that the first nondegeneracy condition of (\ref{1.2}) is actually needed in proving the global exsitence of 2D multispeed Klein-Gordon system with small initial data (See \cite{y2}). We construct a collection of special Klein-Gordon dispersive relations and by iterating the corresponding profiles we can find a blowup in finite time.
\end{abstract}

\section{Introduction}
In this article, we will consider a system of quasilinear, multispeed Klein-Gordon equations in space dimension two, namely
\begin{align}
    \left(\partial^2_t-c^2_\alpha \Delta+b^2_\alpha\right) u_\alpha=\mathcal{Q}_\alpha \left(u,\partial u,\partial^2 u\right),\ \ \ \ 1\le\alpha\le d, \tag{1.1} \label{1.1}
\end{align} 
where the speeds $c_\alpha>0$ and the masses $b_\alpha\in\mathbb{R}-\left\{0\right\}$ are arbitrary positive parameters, and $\mathcal{Q}$ is a quadratic, quasilinear nonlinearity satisfying a suitable symmetry condition (See (1.2) in \cite{y1}). We will show that there exists such a 2D multispeed Klein-Gordon system with a small initial data that it does not have a global solution, namely we expect the solution to blow up in finite time. Then, this result tells us that in order to make sure that the solution to 2D multispeed Klein-Gordon system is always global, we have to impose some restrictions on the coefficients $b_\alpha$ and $c_\alpha$.\par
In 2014, Ionescu and Pausader \cite{a1} obtained the global results of multispeed Klein-Gordon system in 3D with two \underline{nondegeneracy conditions}
\begin{align*}
        \begin{cases}
        b_\sigma-b_\mu-b_\nu\neq 0 \\
        (c_\mu-c_\nu)(c_\mu^2b_\nu-c_\nu^2b_\mu)\ge 0    \end{cases}\ \ \ \ \ \ \ \ \mbox{for any }\sigma,\mu,\nu\in\left\{1,\dots,d\right\}\tag{1.2}.\label{1.2}
\end{align*}
It turns out that these two conditions are redundant in 3D case later on. In 2017, Deng \cite{y1} removed these two nondegeneracy conditions by further exploiting the rotation vector fields $\Omega$. With such rotation vector fields $\Omega$, one is able to restrict the main part of the integral to the region where $\left|\sin\angle\xi,\eta\right|\lesssim 2^{-(\frac{1}{2}-\delta)m}$ ($t\sim 2^m$ and $0<\delta\ll 1$), so with this additional gain, the Z-norm argument is able to be closed as desired.\par
However, in 2D case, things are different and difficult. It's well known that the basic dispersive estimate gives us 
$$\left\|P_k e^{-it\Lambda}f\right\|_{L^\infty}\lesssim (1+t)^{-\frac{n}{2}}(1+2^{2k})\left\|f\right\|_{L^1},$$
where $n$ refers to the space dimension. This means that the decay of the solution of the Klein-Gordon system is less in 2D than in 3D, which will make the problem more difficult. In 2023, we proved the global result in 2D case with two nondegeneracy conditions (\ref{1.2}). Now, this paper will follow Deng's idea \cite{y1} to illustrate that at least the first nondegeneracy condition of (\ref{1.2}) $b_\sigma-b_\mu-b_\nu\neq 0$ is actually needed. \par
For convenience, we first define, for $\forall x\in\mathbb{R}$, 
$$\mbox{sgn } x\triangleq\begin{cases}
+1\ \ \ \ \ \mbox{, if }x>0\\
0\ \ \ \ \ \ \ \mbox{, if }x=0\\
-1\ \ \ \ \ \mbox{, if }x<0\end{cases}.$$
We next denote 
\begin{align*}
    \Phi_{\sigma\mu\nu}(\xi,\eta)&=(-1)^{(\text{sgn} b_\sigma+3)/2}\sqrt{c_\sigma^2\left|\xi\right|^2+b_\sigma^2}-(-1)^{(\text{sgn} b_\mu+3)/2}\sqrt{c_\mu^2\left|\xi-\eta\right|^2+b_\mu^2}\\
    &\ \ \ \ \ \ \ -(-1)^{(\text{sgn} b_\nu+3)/2}\sqrt{c_\nu^2\left|\eta\right|^2+b_\nu^2},\tag{1.3}\label{1.3}
\end{align*}
where $c_\sigma,c_\mu,c_\nu>0$ and $b_\sigma,b_\mu,b_\nu\in\mathbb{R}-\left\{0\right\}$.
Recall that if restricting the nonlinear term to be $u^2$ and letting the initial datum to be $0$ in (\ref{1.1}), then by iteration, the Duhamel's formula will give us that
\begin{align*}
    \widehat{f^\sigma}(t,\xi)=\int_0^t \int_{\mathbb{R}^2} e^{is\Phi_{\sigma\mu\nu}(\xi,\eta)}\widehat{f^\mu}(\xi-\eta)\widehat{f^\nu}(\eta)\,d\eta ds, \tag{1.4}\label{1.4}
\end{align*}
where $\widehat{f^*}\triangleq e^{it\Lambda_*}u_*,\ \ \ *\in\left\{\sigma,\mu,\nu\right\}$ are so-called profiles. The main idea here is to (at least formally) focus on and estimate $\left\|\widehat{f^\sigma}\right\|_{L^2}$. From now on, we will abbreviate $\Phi_{\sigma\mu\nu}$ to $\Phi$ if no ambiguity.\par

 In \cite{a1}, we see that these two nondegeneracy conditions (\ref{1.2}) are used to guarantee that if $\Phi(0,0)=0$ and $\nabla_\eta\Phi(0,0)=0$, then $\det(\nabla_{\eta\eta}\Phi(0,0))\neq 0$. Thus, we first need to try to find a pair of numbers $c_\sigma,c_\mu,c_\nu,b_\sigma,b_\mu,b_\nu$ such that $\Phi(0,0)=0$, $\nabla_\eta\Phi(0,0)=0$, and $\det(\nabla_{\eta\eta}\Phi(0,0))=0$. From $\Phi(0,0)=0$, we get that 
$$(-1)^{(\text{sgn} b_\sigma+3)/2}\left|b_\sigma\right|-(-1)^{(\text{sgn} b_\mu+3)/2}\left|b_\mu\right|-(-1)^{(\text{sgn} b_\nu+3)/2}\left|b_\nu\right|=0.$$
For simplicity, we may pick $b_\sigma=1,b_\mu=2,b_\nu=-1$ as a try. In this case, with some elementary calculations, we can get 
\begin{align*}
    \left(\nabla_{\eta\eta}\Phi\right)(0,0)=\begin{pmatrix}
        -\frac{c_\mu^2}{2}+c_\nu^2 & 0\\
        0 & -\frac{c_\mu^2}{2}+c_\nu^2
    \end{pmatrix}
\end{align*}
Therefore, for simplicity, we may pick $c_\sigma^2=c_\nu^2=1,c_\mu^2=2$ as a try again and thus we consider
\begin{align*}
    \Phi(\xi,\eta)&=\sqrt{\left|\xi\right|^2+1}-\sqrt{2\left|\xi-\eta\right|^2+4}+\sqrt{\left|\eta\right|^2+1} \\
    &=\frac{1}{2}\left\langle \xi,\eta\right\rangle+O\left(\left|\xi\right|^4+\left|\eta\right|^4\right),
\end{align*}
which is exactly (\ref{2.2}) later on. Fortunately, it turns out that with such parameters $c_\sigma,c_\mu,c_\nu,b_\sigma,b_\mu,b_\nu$, we are able to construct a illposed solution that will blowup eventually in time.\par
To formulate it rigorously, by letting $\widetilde{u}_\sigma=(\partial_t-i\Lambda_\sigma)u_\sigma$ and $\widetilde{u}=(\widetilde{u}_\sigma)_{\sigma\in\left\{-1,-2,\dots,-d\right\}\cup\left\{1,2,\dots,d\right\}}$, we can reduce (\ref{1.1}) to 
$$(\partial_t+i\Lambda_\sigma)\widetilde{u}_\sigma=\mathcal{N}_\sigma(\widetilde{u},\widetilde{u}),\ \ \ \ \ \ \ \sigma\in\left\{-1,-2,\dots,-d\right\}\cup\left\{1,2,\dots,d\right\},$$
where $\mathcal{N}_\sigma$ is some quasilinear quadratic term. Now, we can consider the following Klein-Gordon system in space dimension two
\begin{align*}
\begin{cases}
    \left(\partial_t +i\Lambda_1\right) u_1=0 \\
    \left(\partial_t +i\Lambda_2\right) u_2=0 \\
    \left(\partial_t +i\Lambda_3\right) u_3=0 \\
    \left(\partial_t +i\Lambda_1\right) u_4=u_2\cdot\widebar{u_1} \\
    \left(\partial_t +i\Lambda_2\right) u_5=u_3\cdot\widebar{u_2} \\
    \left(\partial_t +i\Lambda_1\right) u_6=u_5\cdot\widebar{u_4} 
\end{cases}, \tag{1.5}\label{1.5}
\end{align*}
where $\widehat{\Lambda_1 f}(\xi)\triangleq\sqrt{\left|\xi\right|^2+1}\cdot\widehat{f}(\xi)$, $\widehat{\Lambda_2 f}(\xi)\triangleq\sqrt{2\left|\xi\right|^2+4}\cdot\widehat{f}(\xi)$ and $\widehat{\Lambda_3 f}(\xi)\triangleq\sqrt{4\left|\xi\right|^2+16}\cdot\widehat{f}(\xi)$. \par
To see the effect of this degeneracy, we consider two inputs of the iteration to be
$$\widehat{f_\mu}(s,\xi-\eta)=2^{cj}\chi(2^j(\xi-\eta)),\ \ \ \ \ \ \ \widehat{f_\nu}(s,\eta)=2^{cj}\chi(2^j\eta),$$
where $\chi$ is a cutoff function, $s$ represents the time and $c$ is some constant. Then, we localize the time in $\left|s
\right|\sim 2^m$. Therefore, in the region where $\left|\eta\right|\sim\left|\xi-\eta\right|\sim 2^{-j}$ and $\left|\xi\right|\sim 2^{-3j}$ with $j=\frac{m}{4}$, we have $\left|\Phi\right|\lesssim 2^{-m}$ due to (\ref{2.2}), so the oscillatory factor $e^{is\Phi}$ is intuitively irrelevant. Thus, using the volume counting to estimate 
 $\left\|\widehat{f_\sigma}\right\|_{L^\infty}$ when its frequency is localized at $\left|\xi\right|\sim 2^{-3j}$, in view of (\ref{1.4}), we see that 
\begin{align*}
    \left\|\widehat{P_{-3j} f_\sigma}\right\|_{L^\infty}&\approx\left\|\int_{2^m}^{2^{m+1}} \int_{\mathbb{R}^2} e^{is\Phi_{\sigma\mu\nu}(\xi,\eta)}\cdot 2^{cj}\chi(2^j(\xi-\eta))\cdot 2^{cj}\chi(2^j\eta)\,d\eta ds\right\|_{L^\infty}\\
    &\lesssim \int_{2^m}^{2^{m+1}} \int_{\left|\eta\right|\lesssim 2^{-j}} 2^{cj}\cdot 2^{cj}\,d\eta ds\lesssim2^m\,2^{-2j}\,2^{2cj}=2^{(3j)(2c+2)/3},
\end{align*}
which implies that a piece of the output function will be of form 
$$2^{(3j)(2c+2)/3}\chi\left(2^{3j}\xi\right).$$
Now, if we start at $c=0$ and do the first iteration, we will see that the first output is $2^{\frac{2}{3}\cdot 3j}\chi\left(2^{3j}\xi\right)$ and the second output is $2^{\frac{10}{9}\cdot 9j}\chi\left(2^{9j}\xi\right)$. Note that $\left\|2^{\frac{10}{9}\cdot 9j}\chi\left(2^{9j}\xi\right)\right\|_{L^2}\approx 2^j\gg 1$. Therefore, it's reasonable for us to focus on the system (\ref{1.5}) and the main contribution where $\left|\eta\right|\sim\left|\xi-\eta\right|\sim 2^{-j}$, $\left|\xi\right|\sim 2^{-3j}$ and $\left|s\right|\sim 2^{4j}$ to find a illposed solution that will blowup eventually in time.\par
Our main result is
\vspace{1em}
\begin{theorem}
Consider the system (\ref{1.5}) in $\left[0,+\infty\right)\times\mathbb{R}_x^2$. Then for any arbitrarily small $\varepsilon>0$, there exists an initial condition $\textbf{u}_0$ such that\par
(i) If $l\le 0$, then $\widehat{P_l\,\textbf{u}_0} \neq 0$;\par
(ii) $\left\|\textbf{u}_0\right\|_{L^2}\lesssim \varepsilon\ll 1$;\par
(iii) Denote $\textbf{u}$ be the corresponding global solution of the Cauchy problem. Then we have 
$$\lim_{t\rightarrow\infty}\left\|\textbf{u}(t)\right\|_{L^2}=+\infty.$$
\end{theorem}
\begin{proof}
    This follows from Chapter 3.
\end{proof}
\vspace{2em}
\begin{remark}
\ \par
In Theorem 1.1, it's reasonable for us to assume that the initial data $\textbf{u}_0$ does not have the high frequency part. This is because in view of the proofs by Z-norm method (\cite{a1}-\cite{y2}) we can always control the high frequency part by using the energy estimate.
\end{remark}
\vspace{2em}
To prove Theorem 1.1, we need to introduce the \underline{profile}:
$$f_\sigma(t)\triangleq e^{it\Lambda_\sigma}u_\sigma(t).$$
We will use Duhamel's formula to express $\widehat{f_4}$ $\left(\widehat{f_5}\right)$ and $\widehat{f_6}$. Next, we view $\widehat{f_4}$ $\left(\widehat{f_5}\right)$ as our first iteration (with $\widehat{f_4}$ be the output and $\widehat{(\textbf{u}_0)_1},\widehat{(\textbf{u}_0)_2}$ be the inputs) and view $\widehat{f_6}$ as our second iteration (with $\widehat{f_6}$ be the output and $\widehat{f_4},\widehat{f_5}$ be the inputs). Then, we estimate the outputs $\widehat{f_4}$ and $\widehat{f_6}$ in turn, and find the blow up in finite time.\par
The plan of this article: in chapter 2, we will collect all notations and elementary lemmas that are needed in the main proof; in chapter 3, we will estimate the outputs of two iterations and obtain a blow up in finite time.
\vspace{1em}
\section{Notations and Lemmas}
In this chapter, we will introduce all notations that will be used later on. In addition, we will study the elementary properties of our phase function.\par
First, to do the Littlewood-Paley projections later on, we define $\varphi:\mathbb{R}\rightarrow\left[0,1\right]$ to be an even smooth function that is supported in $\left[-1.2,1.2\right]$ and equals 1 in $\left[-1.1,1.1\right]$. Abusing the notation, we also write $\varphi:\mathbb{R}^2\rightarrow\left[0,1\right]$ to be the corresponding radial function on $\mathbb{R}^2$. Let
\begin{align*}
    \varphi_k(x)&\triangleq \varphi\left(\left|x\right|/2^k\right)-\varphi\left(\left|x\right|/2^{k-1}\right)\ \ \ \ \ \mbox{for any }k\in\mathbb{Z},\\
    \varphi_I&\triangleq \sum_{m\in I\cap\mathbb{Z}} \varphi_m \ \ \ \ \ \mbox{for any }I \subseteq\mathbb{R}.
\end{align*}
For any $c\in\mathbb{R}$ let
$$\varphi_{\le c}\triangleq\varphi_{\left(-\infty,c\right]},\ \ \varphi_{\ge c}\triangleq\varphi_{\left[c,+\infty\right)},\ \ \varphi_{< c}\triangleq\varphi_{\left(-\infty,c\right)},\ \ \varphi_{> c}\triangleq\varphi_{\left(c,+\infty\right)}.$$
For any $a<b\in\mathbb{Z}$ and $j\in\left[a,b\right]\cap\mathbb{Z}$ let
\begin{align*}
    \varphi_j^{\left[a,b\right]}\triangleq\begin{cases}
        \varphi_j&\mbox{if }a<j<b,\\
        \varphi_{\le a}&\mbox{if }j=a,\\
        \varphi_{\ge b}&\mbox{if }j=b.
    \end{cases}
\end{align*}\par
Let $P_k$, $k\in\mathbb{Z}$, denote the operator on $\mathbb{R}^2$ defined by the Fourier multiplier $\xi\rightarrow\varphi_k(\xi)$. Let $P_{\le c}$ (respectively $P_{>c}$) denote the operators on $\mathbb{R}^2$ defined by the Fourier multipliers $\xi\rightarrow\varphi_{\le c}(\xi)$ (respectively $\xi\rightarrow\varphi_{>B}(\xi)$). \par
\vspace{2em}
Next, recall (\ref{1.3}) and pick $c_\sigma^2=c_\nu^2=1,c_\mu^2=2,b_\sigma=1,b_\mu=2,b_\nu=-1$, we define the phase function
\begin{align*}
    \Phi(\xi,\eta)&\triangleq\sqrt{\left|\xi\right|^2+1}-\sqrt{2\left|\xi-\eta\right|^2+4}+\sqrt{\left|\eta\right|^2+1} \\
    &=\frac{1}{2}\left\langle \xi,\eta\right\rangle+O\left(\left|\xi\right|^4+\left|\eta\right|^4\right). \tag{2.2}\label{2.2}
\end{align*}
Let $\eta=2^{-\lambda l}\,\eta^\prime$, $\xi=2^{-3l}\,\xi^\prime$, and we can rewrite
\begin{align*}
    \Phi(\xi,\eta)&=\Phi(2^{-3l}\xi^\prime, 2^{-\lambda l}\eta^\prime) \\
    &=2^{-4\lambda l}\left[2^{4\lambda l}\left(\sqrt{1+\left|2^{-3l}\xi^\prime\right|^2}-\sqrt{2\left|2^{-3l}\xi^\prime-2^{-\lambda l}\eta^\prime\right|^2+4}+\sqrt{\left|2^{-\lambda l}\eta^\prime\right|^2+1}\right)\right] \\
    &\triangleq 2^{-4\lambda l}\,\widetilde{\Phi_\lambda}(\xi^\prime,\eta^\prime).\tag{2.3}\label{2.3}
\end{align*}
For simplicity, we just write $\widetilde{\Phi}\triangleq\widetilde{\Phi_1}$. In the following lemmas, we will study the zeros and critical points of the rescaled phase function $\widetilde{\Phi_\lambda}$.
\vspace{0.8em}
\begin{lemma}
    Fix $\delta=0.01\ll 1$. Assume $\left|\xi\right|\sim 2^{-3l}$ and $\left|\eta\right|\sim 2^{-\lambda l}$, where $0\le \lambda\le 1-\delta$. When $l> D\ge 1000\gg 1$, $\nabla_{\eta^\prime}\widetilde{\Phi_\lambda}(\xi^\prime,\eta^\prime)\neq 0$.
\end{lemma}
\begin{proof}
We first calculate the derivative
\begin{align*}
    \nabla_{\eta^\prime}\widetilde{\Phi_\lambda}(\xi^\prime,\eta^\prime)=2^{4\lambda l}\left(\frac{2\cdot 2^{-\lambda l}\left(2^{-3l}\xi^\prime-2^{-\lambda l}\eta^\prime\right)}{\sqrt{2\left|2^{-3l}\xi^\prime-2^{-\lambda l}\eta^\prime\right|^2+4}}+\frac{2^{-2\lambda l}\eta^\prime}{\sqrt{\left|2^{-\lambda l}\eta^\prime\right|^2+1}}\right).
\end{align*}
We prove by contradiction. If the claim is not true, then there exists $\xi^\prime, \eta^\prime$, such that $\left|\xi^\prime\right|,\left|\eta^\prime\right|\sim 1$ and 
$$\frac{2\cdot 2^{-\lambda l}\left(2^{-3l}\xi^\prime-2^{-\lambda l}\eta^\prime\right)}{\sqrt{2\left|2^{-3l}\xi^\prime-2^{-\lambda l}\eta^\prime\right|^2+4}}+\frac{2^{-2\lambda l}\eta^\prime}{\sqrt{\left|2^{-\lambda l}\eta^\prime\right|^2+1}}=0.$$
We next go back to $\xi,\eta$ variables
$$\frac{2\left(\xi-\eta\right)}{\sqrt{2\left|\xi-\eta\right|^2+4}}+\frac{\eta}{\sqrt{\left|\eta\right|^2+1}}=0.$$
Then, we see that $\xi\parallel\eta$. Therefore, we may suppose $\xi=(a,0)$ and $\eta=(b,0)$, where $\left|a\right|\sim 2^{-3l}$ and $\left|b\right|\sim 2^{-\lambda l}$. Now we only need to consider
\begin{align*}
    \frac{2\left(a-b\right)}{\sqrt{2(a-b)^2+4}}+\frac{b}{\sqrt{b^2+1}}=0.\tag{2.4}\label{2.4}
\end{align*}
Let $g(b)=\displaystyle{\frac{2b}{\sqrt{2b^2+4}}}$, then we see that $g^\prime(b)=\displaystyle{\frac{2\sqrt{2}}{\left(b^2+2\right)^{\frac{3}{2}}}}$. Therefore, $\left|g^\prime(b)\right|\le 1$, which implies that 
$\left|g(a-b)-g(-b)\right|\le \left|a\right|$. Then, (\ref{2.4}) gives
\begin{align*}
    \left|-\frac{2b}{\sqrt{2b^2+4}}+\frac{b}{\sqrt{b^2+1}}\right|\le\left|a\right|.\tag{2.5}\label{2.5}
\end{align*}
Let $\displaystyle{h(b)=-\frac{2b}{\sqrt{2b^2+4}}+\frac{b}{\sqrt{b^2+1}}\le 0}$. By some computation, we get that $h^\prime(0)=h^{\prime\prime}(0)=0$, 
$$h^\prime(b)=-\frac{2\sqrt{2}}{(b^2+2)^{\frac{3}{2}}}+\frac{1}{(b^2+1)^{\frac{3}{2}}}\le 0$$
$$h^{\prime\prime\prime}(b)=-\frac{24(2b^2-1)}{\sqrt{2}(b^2+2)^{\frac{7}{2}}}-\frac{3(-4b^2+1)}{(b^2+1)^{\frac{7}{2}}}$$
and
$$h^{(4)}(b)=-\frac{60\sqrt{2}(2b^2-3)}{(b^2+2)^{\frac{9}{2}}}+\frac{15b(-4b^2+3)}{(b^2+1)^{\frac{9}{2}}}.$$
Note\vspace{0.4em} that if $0\le b\le 0.3$, then $h^{(4)}(b)\ge 0$ since two terms of $h^{(4)}$ are both positive, which implies that $h^{\prime\prime\prime}$ is increasing when $0\le b\le 0.3$. Also note that $h^{\prime\prime\prime}(0)=-1.5$, $h$ is an odd function and $h^{\prime\prime\prime}$ is an even function. Now, we consider the following two cases.\par
If $\lambda l\le 2$, then $\min\limits_{\left|b\right|\ge 0.55\times 2^{-2}}\left|h(b)\right|=-h(0.55\times 2^{-2})> 0.0006$. However, we also \vspace{0.4em} that $\left|a\right|\le 1.2\times 2^{-3l}\le 1.2 \times 2^{-3000}\ll 0.0006$, which leads to a contradiction\vspace{0.4em} in view of (\ref{2.5}).\par
On the other hand, if $\lambda l\ge 2$, then $\min\limits_{\left|b\right|\sim 2^{-\lambda l}}\left|h^{\prime\prime\prime}(b)\right|=-h^{\prime\prime\prime}(1.2\times 2^{-2})=-h^{\prime\prime\prime}(0.3)\ge 1$. By Taylor expansion near $b=0$ and using the remainders of Lagrange form, we can write $\restr{h}{\left(-0.3,0.3\right)}(b)=\frac{h^{\prime\prime\prime}(\Tilde{b})}{6}b^3$, where $\Tilde{b}$ is a number between $0$ and $b$. Therefore, we get
\begin{align*}
   \min\limits_{\left|b\right|\sim 2^{-\lambda l}}\left|h(b)\right|&=-h(0.55\times 2^{-\lambda l})\ge \frac{1}{6}\cdot (0.55\times 2^{-\lambda l})^3\ge 0.027\cdot 2^{-3\lambda l}\\
   &>0.027\cdot 2^{30}\cdot 2^{-3l}\gg 1.201\cdot 2^{-3l}> \left|a\right|,
\end{align*}
which again leads to a contradiction in view of (\ref{2.5}).\par
Our proof is finished.
\end{proof}
\vspace{0.8em}
\begin{cor}
    Fix $\delta=0.01\ll 1$. Assume $\left|\xi\right|\sim 2^{-3l}$ and $\left|\eta\right|\ge 2^{-l+\delta l}$. When $l> D\ge 1000\gg 1$, $\nabla_{\eta^\prime}\widetilde{\Phi}(\xi^\prime,\eta^\prime)\neq 0$.
\end{cor}
\begin{proof}
    Same as Lemma 2.1.
\end{proof}
\vspace{0.8em}
\begin{cor}
Let $l>D\gg 1$. Assume $\left|\xi\right|\sim 2^{-3l}$, then there exists a unique $\eta$ such that $\left|\eta\right|\sim 2^{-l}$ and $\nabla_{\eta^\prime}\widetilde{\Phi}(\xi^\prime,\eta^\prime)=0$. Moreover, we have $\widetilde{\Phi}(\xi^\prime,\eta^\prime)\neq 0$.
\end{cor}
\begin{proof}
We observe that 
$$\nabla_{\eta^\prime}\widetilde{\Phi}(\xi^\prime,\eta^\prime)=0\ \ \ \iff \ \ \ \nabla_\eta\Phi(\xi,\eta)=0.$$
So we only need to consider
$$\nabla_\eta\Phi(\xi,\eta)=\frac{2\left(\xi-\eta\right)}{\sqrt{2\left|\xi-\eta\right|^2+4}}+\frac{\eta}{\sqrt{\left|\eta\right|^2+1}}=0.$$
Then, we again see that $\xi\parallel\eta$. Therefore, we may suppose $\xi=(a,0)$ and $\eta=(b,0)$, where $\left|a\right|\sim 2^{-3l}$. Now we only need to consider
\begin{align*}
    \frac{2\left(a-b\right)}{\sqrt{2(a-b)^2+4}}+\frac{b}{\sqrt{b^2+1}}=0.
\end{align*}
WLOG, we may assume $a>0$. Then by some elementary computation, we get that
\begin{align*}
    a=b-\sqrt{\frac{2b^2}{b^2+2}}.\tag{2.6}\label{2.6}
\end{align*}
If $b>0.001$, then denote $f(b)=b-\sqrt{\frac{2b^2}{b^2+2}}$ and we get $f^\prime(b)=1-\frac{2\sqrt{2}}{(b^2+2)^{\frac{2}{3}}}> 0$ . This means that $a=f(b)$ is a strictly increasing function of $b$, so $a\gg 2^{-3l}$ (since we've assumed $l>D\gg 1$). Therefore we must have $0<b\ll 1$, which implies that we can do the Taylor expansion of $f(b)$ around $b=0$. Now, (\ref{2.6}) gives that
$$a=b-\left(b+O\left(b^3\right)\right)=O\left(b^3\right).$$
This means that there exists a $\eta$ such that (\ref{2.6}) holds and $\left|\eta\right|\sim 2^{-l}$.\par
Next, we show that $\widetilde{\Phi}(\xi^\prime,\eta^\prime)\neq 0$. If not, then we have $\widetilde{\Phi}(\xi^\prime,\eta^\prime)=0$, which implies that $\Phi(\xi,\eta)=0$. However, when $\left|\xi\right|,\left|\eta\right|\ll 1$, we have $\Phi(\xi,\eta)\approx\left\langle\xi,\eta\right\rangle=\left|\xi\right|\cdot\left|\eta\right|=2^{-4l}\neq 0$ (note that $\xi\parallel\eta$), which is a contradiction. 
\end{proof}
\vspace{1em}
\section{Proof of the Main Theorem}
In this chapter, we will iterate twice and extract the main part to prove the main theorem.\par
Define $\widehat{g_0}(\xi)=\varepsilon\cdot \varphi_{\left(-\infty,0\right]}(\xi)$, which gives us that $\left\|g_0\right\|_{L^2}\lesssim \varepsilon$. Moreover, we see that if $\left|\xi\right|\le 1.1$, then $\widehat{g_0}\equiv 1$ and if $\left|\xi\right|\ge 1.2$, then $\widehat{g_0}\equiv 0$. Now take the initial value of the system (\ref{1.1}) to be
\begin{align*}
    \textbf{u}_0(x)=\begin{pmatrix}
        g_0 (x) \\
        g_0 (x) \\
        g_0 (x) \\
        0 \\
        0 \\
        0
    \end{pmatrix}.
\end{align*}
In the following proof, we fix $D\gg 1$. For example, $D=10000$ works in view of the following proof. \par
First, we consider the first iteration. In Proposition 3.1, we assume the frequency of $\widehat{f_4}$ (and $\widehat{f_5}$) is very low. We will find the pointwise size of $\widehat{P_{-3l}f_4}$ (and $\widehat{P_{-3l}f_5}$), and point out that the main contribution would occur at $\left|\xi\right|\sim 2^{-3l}$, $\left|\eta\right|\sim 2^{-l}$ and $t\sim 2^{4l}$.
\vspace{0.8em}
\begin{prop}
    For any $l\ge D\ge 0$, we fix $l$ and consider the function $\widehat{P_{-3l} f_4}(\xi,t)$. Then, we have the following\par
    (a) For any $t\ge 2^{4.05l}$ and $\left|\xi\right|\sim 2^{-3l}$, $\left|\widehat{P_{-3l} f_4}(\xi,t)\right|\approx 2^{2l}$. Moreover, we can write 
    $$\widehat{P_{-3l}f_4}(\xi,t)\approx\reallywidehat{\mathcal{B}}_{\left[3.99l,4.05l\right],\left[-1.03l,-0.99l\right]}+e(\xi,t)\approx 2^{2l}+e(\xi,t),$$
    where $\left|e(\xi,t)\right|\lesssim 2^{1.997l}$.\par
    (b) For any $t\le 2^{4.05l}$ and $\left|\xi\right|\sim 2^{-3l}$, $\left|\widehat{P_{-3l} f_4}(\xi,t)\right|\lesssim 2^{2l}$.
\end{prop}
\begin{proof}
\ \par
(a) In this proof, we set $\xi^\prime=2^{3l}\xi$ and $t\sim 2^m$. The main observation in our proof is the main term of the Duhamel's formula of $\widehat{P_{-3l}f_4}(\xi,t)$ is $\reallywidehat{\mathcal{B}}_{\left[3.99l,4.05l\right],\left[-1.03l,-0.99l\right]}$.\par
In fact, we can write
\begin{align*}
    &\reallywidehat{\mathcal{B}}_{\left[3.99l,4.05l\right],\left[-1.03l,-0.99l\right]}\\
    =&\reallywidehat{\mathcal{B}}_{\left[3.99l,4.05l\right],\left[-1.03l,-0.99l\right]}(2^{-3l}\xi^\prime,t)=\reallywidehat{\mathcal{B}}_{\left[3.99l,4.05l\right],\left[-1.03l,-0.99l\right]}(\xi,t)\\
    \triangleq&\int_{2^{3.99l}}^{2^{4.05l}}\int_{\mathbb{R}^2} e^{is\Phi(\xi,\eta)} \widehat{g_0}(\xi-\eta)\widehat{g_0}(\eta)\mathbbm{1}_{\left[-1.03l,-0.99l\right]}(\eta)\,d\eta ds \\
    =&\int_{2^{3.99l}}^{2^{4.05l}}\int_{2^{-1.03l}\le\left|\eta\right|\le 2^{-0.99l}} e^{is\Phi(\xi,\eta)} \,d\eta ds \tag{3.1a}\label{3.1a} \\
    =&\int_{2^{3.99l}}^{2^{4.05l}}\int_{2^{-1.03l}\le\left|\eta\right|\le 2^{-0.99l}} e^{is\Phi(\xi,\eta)}\widehat{g_0}(\xi-\eta)\widehat{g_0}(\eta) \,d\eta ds \tag{3.1b}\label{3.1b}\\
\end{align*}
Now, we consider the contribution of the above integral (\ref{3.1a}) when $\left|\eta\right|\notin\left[-1.03l,-0.99l\right]$. In fact, we have
\begin{align*}
    &\int_{2^{3.99l}}^{2^{4.05l}}\int_{\mathbb{R}^2} e^{is\Phi(\xi,\eta)} \,d\eta ds \\
    =&\int_{2^{3.99l}}^{2^{4.05l}}\int_{\left|\eta\right|\le 2^{-1.03l}} e^{is\Phi(\xi,\eta)} \,d\eta ds+\int_{2^{3.99l}}^{2^{4.05l}}\int_{2^{-1.03l}\le \left|\eta\right|\le 2^{-0.99l}} e^{is\Phi(\xi,\eta)} \,d\eta ds\\
    &\ \ \ \ +\int_{2^{3.99l}}^{2^{4.05l}}\int_{\left|\eta\right|\ge 2^{-0.99l}} e^{is\Phi(\xi,\eta)} \,d\eta ds\triangleq I_1+I_2+I_3.
\end{align*}
As for $I_3$, we can decompose in terms of the length of $\eta$ 
$$I_3=\sum_{k_2>-0.99l}\int_{2^{3.99l}}^{2^{4.05l}}\int_{\left|\eta\right|\sim 2^{k_2}} e^{is\Phi(\xi,\eta)}\,d\eta ds.$$
Then, let $\eta^\prime=2^{\lambda l}\eta$ and $s^\prime=2^{-4l}s$ ($0\le\lambda \le 0.99$) and get
\begin{align*}
    I_3=\sum_{k_2>-0.99l} 2^{4l-2\lambda l} \int_{2^{-0.01l}}^{2^{0.05l}}\int_{\left|\eta^\prime\right|\sim 1}e^{i\left(s^\prime\cdot 2^{(4-4\lambda)l}\right)\widetilde{\Phi_\lambda}(\xi^\prime,\eta^\prime)} \,d\eta^\prime ds^\prime.
\end{align*}
Note that $\left|\nabla_{\eta^\prime}\widehat{g_0}(2^{-3l}\xi^\prime-2^{-\lambda l}\eta^\prime)\right|\sim 2^{-\lambda l}$, $\left|\nabla_{\eta^\prime}\widehat{g_0}(2^{-\lambda l}\eta^\prime)\right|\sim 2^{-\lambda l}$ and $s^\prime\cdot 2^{(4-4\lambda)l}\ge 2^{0.03l}$. 
Lemma 2.1 tells us that $\nabla_{\eta^\prime}\widetilde{\Phi_\lambda}(\xi^\prime,\eta^\prime)\neq 0$, which implies that
\begin{align*}
    I_3&=\sum_{k_2>-0.99l} 2^{4l-2\lambda l} \int_{2^{-0.01l}}^{2^{0.05l}} O\left(\left(s^\prime\cdot 2^{(4-4\lambda)l}\right)^{-N}\right) ds^\prime
\end{align*}
So, by take $N,D$ large enough, we can get 
\begin{align*}
    \left|I_3\right|\le\sum_{k_2>-0.99l} A_N 2^{4l-2\lambda l-0.04Nl+0.05l}=\frac{4}{3} A_N 2^{(4.05-0.04N)l}\le 2^{-10l}
\end{align*}
On the other hand, as for $I_1$, let $\eta^\prime=2^l \eta$ and $s^\prime=2^{-4l}s$ and we get
\begin{align*}
    I_1=2^{2l} \int_{2^{-0.01l}}^{2^{0.05l}} \int_{\left|\eta^\prime\right|\le 2^{-0.03l}}e^{is^\prime\widetilde{\Phi}(\xi^\prime,\eta^\prime)}\,d\eta^\prime ds^\prime.
\end{align*}
So, by volume estimate, we get 
\begin{align*}
    \left|I_1\right|\le \pi\cdot 2^{1.99l}.
\end{align*}
Thus, in view of (\ref{3.1a}), we have
\begin{align*}
    \reallywidehat{\mathcal{B}}_{\left[3.99l,4.05l\right],\left[-1.03l,-0.99l\right]}= \int_{2^{3.99l}}^{2^{4.05l}} \int_{\mathbb{R}^2} e^{is\Phi(\xi,\eta)}\,d\eta ds+O\left(2^{1.99l}\right)
\end{align*}
Let $\eta^\prime=2^l \eta$ and $s^\prime=2^{-4l} s$. Then, we get
\begin{align*}
    \reallywidehat{\mathcal{B}}_{\left[3.99l,4.05l\right],\left[-1.03l,-0.99l\right]}&= 2^{2l}\cdot\int_{2^{-0.01l}}^{2^{0.05l}} \int_{\mathbb{R}^2} e^{is^\prime\widetilde{\Phi}(\xi^\prime,\eta^\prime)} \,d\eta^\prime ds^\prime+O\left(2^{1.99l}\right) \\
    &\triangleq 2^{2l}\cdot\int_{2^{-0.01l}}^{2^{0.05l}}G(s^\prime)\,ds^\prime+O\left(2^{1.99l}\right).\tag{3.2}\label{3.2}
\end{align*}
Now, if $\left|s^\prime\right|\le 2^{-0.01l}$, then we decompose
\begin{align*}
    G(s^\prime)&=\int_{\mathbb{R}^2} e^{is^\prime\widetilde{\Phi}(\xi^\prime,\eta^\prime)} \,d\eta^\prime\\
    &=\int_{\left|\eta^\prime\right|\le\left(s^\prime\right)^{-\frac{1}{3}-0.01}} e^{is^\prime\widetilde{\Phi}(\xi^\prime,\eta^\prime)} \,d\eta^\prime\\
    &\ \ \ +\int_{\left|\eta^\prime\right|\ge\left(s^\prime\right)^{-\frac{1}{3}-0.01}} e^{is^\prime\widetilde{\Phi}(\xi^\prime,\eta^\prime)} \,d\eta^\prime\\
    &\triangleq I_{11}+I_{12}.
\end{align*}
By volume estimate, we see that
\begin{align*}
    \left|I_{11}\right|\le\pi\cdot\left(s^\prime\right)^{-\frac{2}{3}-0.02}.\tag{3.3}\label{3.3}
\end{align*}
In terms of $I_{12}$, we let $s^\prime\sim 2^{-\lambda l}$ ($\lambda\ge 0.01$) and write 
$$I_{12}=\sum_{\frac{\lambda}{3}+0.01\lambda<\Tilde{\lambda}\le 1 \atop \Tilde{\lambda}l\in\mathbb{Z}} \int_{\left|\eta^\prime\right|\sim 2^{\Tilde{\lambda}l}} e^{is^\prime \Tilde{\Phi}(\xi^\prime,\eta^\prime)}\,d\eta^\prime.$$
Then, note that $\nabla_{\eta^\prime}\left[s^\prime \Tilde{\Phi}(\xi^\prime,\eta^\prime)\right]=s^\prime\left(\frac{1}{2}\xi^\prime+O(\left|\eta^\prime\right|^2\eta^\prime)\right)$, which implies that $\left|\nabla_{\eta^\prime}\left[s^\prime \Tilde{\Phi}(\xi^\prime,\eta^\prime)\right]\right|\ge 2^{-\lambda l}\cdot 2^{3\Tilde{\lambda}l}\ge 2^{0.03\lambda l}$. Also note that $\left|\nabla_{\eta^\prime}\widehat{g_0}(2^{-3l}\xi^\prime-2^{-l}\eta^\prime)\right|\sim 2^{-l}$ and  $\left|\nabla_{\eta^\prime}\widehat{g_0}(2^{-l}\eta^\prime)\right|\sim 2^{-l}$. Thus, we can keep doing integration by parts in $\eta^\prime$ such that for all $\frac{\lambda}{3}+0.01\lambda<\Tilde{\lambda}<1$, we have
$$\left|\int_{\left|\eta^\prime\right|\sim 2^{\Tilde{\lambda}l}} e^{is^\prime \Tilde{\Phi}(\xi^\prime,\eta^\prime)}\,d\eta^\prime\right|\le 2^{-11\lambda l}.$$
This gives
\begin{align*}
    \left|I_{12}\right|\le \sum_{\frac{\lambda}{3}+0.01\lambda<\Tilde{\lambda}\le 1 \atop \Tilde{\lambda}l\in\mathbb{Z}} 2^{-11\lambda l} \le \left|l\right|\cdot 2^{-11\lambda l}\le \frac{1}{\lambda}\cdot\left|\lambda l\right|\cdot 2^{-11\lambda l}\le \frac{1}{\lambda}\cdot 2^{-10\lambda l}\le 100\cdot \left(s^\prime\right)^{10}\le \pi \cdot \left(s^\prime\right)^{-\frac{2}{3}-0.02}.\tag{3.4}\label{3.4}
\end{align*}
Combining (\ref{3.3}) and (\ref{3.4}), we conclude
\begin{align*}
    \left|\int_0^{2^{-0.01l}} G(s^\prime)\,ds^\prime\right|\le \pi\cdot\int_0^{2^{-0.01l}} (s^\prime)^{-\frac{2}{3}-0.02}\,ds^\prime=\frac{150\pi}{47}\cdot 2^{-\frac{47}{15000}l}
\end{align*}
On the other hand, we suppose that $\left|s^\prime\right|\ge 2^{0.05l}$. Note that $\widehat{g_0}$ has a compact support. By Corollary 2.3, we know that there exists $\eta^\prime(\xi^\prime)$ such that $\nabla_{\eta^\prime}\Tilde{\Phi}(\xi^\prime,\eta^\prime(\xi^\prime))=0$ and $\left|2^{-l}\eta^\prime(\xi^\prime)\right|\sim 2^{-l}$,, then we can write
\begin{align*}
    G(s^\prime)=C e^{is^\prime \Tilde{\Phi}(\xi^\prime,\eta^\prime(\xi^\prime))}\cdot\frac{1}{s^\prime}+O\left(\frac{1}{(s^\prime)^2}\right).
\end{align*}
Therefore,
\begin{align*}
    \left|\int_{2^{0.05l}}^{+\infty} G(s^\prime)\,ds^\prime\right|&\lesssim \left|\int_{2^{0.05l}}^{+\infty} \frac{e^{is^\prime\Tilde{\Phi}(\xi^\prime,\eta^\prime(\xi^\prime))}}{s^\prime}\,ds^\prime\right|+\int_{2^{0.05l}}^{+\infty} \frac{ds^\prime}{(s^\prime)^2}\\
    &\lesssim \left|\int_{2^{0.05l}}^{+\infty} \frac{e^{is^\prime\Tilde{\Phi}(\xi^\prime,\eta^\prime(\xi^\prime))}}{ (s^\prime)^2}\,ds^\prime\right|+\frac{e^{i\cdot 2^{0.05l}\cdot \Tilde{\Phi}(\xi^\prime,\eta^\prime(\xi^\prime))}}{2^{0.05l}}+2^{-0.05l}\\
    &\lesssim 2^{-0.05l}.\tag{3.5}\label{3.5}
\end{align*}
Finally, in view of (\ref{3.2}), we see that
\begin{align*}
    \reallywidehat{\mathcal{B}}_{\left[3.99l,4.05l\right],\left[-1.03l,-0.99l\right]}=C_\varphi \cdot 2^{2l}\cdot\int_0^{+\infty}G(s^\prime)\,ds^\prime+O\left(2^{1.997l}\right).\tag{3.6a}\label{3.6a}
\end{align*}
Note that
\begin{align*}
    \int_0^{+\infty}G(s^\prime)\,ds^\prime&=\int_0^{+\infty}\int_{\mathbb{R}^2} e^{is^\prime\widetilde{\Phi}(\xi^\prime,\eta^\prime)} \,d\eta^\prime ds^\prime\neq 0
\end{align*}
is a nonzero constant independent of $l$.\par
By a same computation, we can also get
\begin{align*}
    \reallywidehat{\mathcal{B}}_{\left[3.99l,4.05l\right],\left[-1.03l,-0.99l\right]}&=2^{2l}\cdot\int_0^{+\infty}\int_{\mathbb{R}^2} e^{is^\prime\widetilde{\Phi}(\xi^\prime,\eta^\prime)} \widehat{g_0}(2^{-3l}\xi^\prime-2^{-l}\eta^\prime)\widehat{g_0}(2^{-l}\eta^\prime) \,d\eta^\prime ds^\prime\\
    &\ \ \ \ +O\left(2^{1.997l}\right).\tag{3.6b}\label{3.6b}
\end{align*}
Combining (\ref{3.6a}) and (\ref{3.6b}), we can get
\begin{align*}
    \int_0^{+\infty}G(s^\prime)\,ds^\prime=\int_0^{+\infty}\int_{\mathbb{R}^2} e^{is^\prime\widetilde{\Phi}(\xi^\prime,\eta^\prime)} \widehat{g_0}(2^{-3l}\xi^\prime-2^{-l}\eta^\prime)\widehat{g_0}(2^{-l}\eta^\prime) \,d\eta^\prime ds^\prime+O\left(2^{-0.01l}\right).
\end{align*}
Namely, we conclude 
\begin{align*}
    \int_0^{+\infty}G(s^\prime)\,ds^\prime\approx\int_0^{+\infty}\int_{\mathbb{R}^2} e^{is^\prime\widetilde{\Phi}(\xi^\prime,\eta^\prime)} \widehat{g_0}(2^{-3l}\xi^\prime-2^{-l}\eta^\prime)\widehat{g_0}(2^{-l}\eta^\prime) \,d\eta^\prime ds^\prime.
\end{align*}
In addition, in view of (\ref{3.5}), if $m-4l\ge 0.05l$, then we also have $\displaystyle{\left|\int_{2^{0.05l}}^{2^{m-4l}} G(s^\prime)\,ds^\prime\right|\le C_3\cdot 2^{-0.05l}}$, which implies that 
$$\displaystyle{\left|\int_{2^{m-4l}}^{+\infty} G(s^\prime)\,ds^\prime\right|\le C_3\cdot 2^{-0.05l}}.$$
Therefore, we yield that
\begin{align*}
    \int_0^{+\infty}G(s^\prime)\,ds^\prime\approx\int_0^{2^{m-4l}}G(s^\prime)\,ds^\prime;
\end{align*}
similarly, we also have
\begin{align*}
    &\int_0^{+\infty}\int_{\mathbb{R}^2} e^{is^\prime\widetilde{\Phi}(\xi^\prime,\eta^\prime)} \widehat{g_0}(2^{-3l}\xi^\prime-2^{-l}\eta^\prime)\widehat{g_0}(2^{-l}\eta^\prime) \,d\eta^\prime ds^\prime\\
    \approx&\int_0^{2^{m-4l}}\int_{\mathbb{R}^2} e^{is^\prime\widetilde{\Phi}(\xi^\prime,\eta^\prime)} \widehat{g_0}(2^{-3l}\xi^\prime-2^{-l}\eta^\prime)\widehat{g_0}(2^{-l}\eta^\prime) \,d\eta^\prime ds^\prime.
\end{align*}
Now, we compute
\begin{align*}
    \widehat{P_{-3l}f_4}(\xi,t)&=\int_0^{2^m}\int_{\mathbb{R}^2} e^{is\Phi(\xi,\eta)} \widehat{g_0}(\xi-\eta)\widehat{g_0}(\eta) \,d\eta ds \\
    &=2^{2l}\cdot\int_0^{2^{m-4l}}\int_{\mathbb{R}^2} e^{is^\prime\widetilde{\Phi}(\xi^\prime,\eta^\prime)} \widehat{g_0}(2^{-3l}\xi^\prime-2^{-l}\eta^\prime)\widehat{g_0}(2^{-l}\eta^\prime) \,d\eta^\prime ds^\prime\\
    &\approx 2^{2l}\cdot \int_0^{+\infty} G(s^\prime)\,ds^\prime \approx 2^{2l}.
\end{align*}
Namely, $\left|\widehat{P_{-3l}f_4}(\xi,t)\right|\approx 2^{2l}$, provided that $m\ge 4.05l$ and $\left|\xi\right|\sim 2^{-3l}$. Moreover, we have that
\begin{align*}
   \widehat{P_{-3l}f_4}(\xi,t)=\reallywidehat{\mathcal{B}}_{\left[3.99l,4.05l\right],\left[-1.03l,-0.99l\right]}+O\left(2^{1.997l}\right).
\end{align*}\par
\vspace{1em}
(b) We divide into three cases: $2^{4l}<t \le 2^{4.05l}$, $2^l<t\le 2^{4l}$ and $t\le 2^l$.\par
\noindent $1^\circ$: $2^{4l}<t \le 2^{4.05l}$\par
In this case, we let $t\sim 2^{(4+\alpha)l}$, where $0<\alpha\le 0.05$. Then, we can proceed as before in part (a) to write that
\begin{align*}
    \widehat{P_{-3l}f_4}(\xi,t)&=\int_0^{2^{(4+\alpha)l}}\int_{\mathbb{R}^2} e^{is\Phi(\xi,\eta)} \widehat{g_0}(\xi-\eta)\widehat{g_0}(\eta) \,d\eta ds \\
    &=2^{2l}\cdot\int_0^{2^{\alpha l}}\int_{\mathbb{R}^2} e^{is^\prime\widetilde{\Phi}(\xi^\prime,\eta^\prime)} \widehat{g_0}(2^{-3l}\xi^\prime-2^{-l}\eta^\prime)\widehat{g_0}(2^{-l}\eta^\prime) \,d\eta^\prime ds^\prime\\
    &\approx 2^{2l}\cdot \int_0^{2^{\alpha l}} G(s^\prime)\,ds^\prime.
\end{align*}
Now, a similar computation as in (\ref{3.5}) tells us that
$$\displaystyle{\left|\int_{2^{\alpha l}}^{+\infty} G(s^\prime)\,ds^\prime\right|\le C_3\cdot 2^{-\alpha l}}\lesssim 1,$$
which implies that
$$\left|\int_0^{2^{\alpha l}} G(s^\prime)\,ds^\prime\right|\le\left|\int_0^{+\infty} G(s^\prime)\,ds^\prime\right|+\left|\int_{2^{\alpha l}}^{+\infty} G(s^\prime)\,ds^\prime\right|\lesssim 1.$$
Therefore, we conclude that
$$\left\|\widehat{P_{-3l}f_4}(\xi,t)\right\|_{L^\infty}\lesssim 2^{2l}.$$\par
\noindent $2^\circ$: $2^l<t \le 2^{4l}$\par
In this case, let $s^\prime=2^{-4l}\cdot s$ and we observe that $s^\prime\le 2^{-3l}$ and therefore $(s^\prime)^{-\frac{1}{3}-0.01}\ge 2^{1.03l}$. Now, we can similarly write
\begin{align*}
    \widehat{P_{-3l}f_4}(\xi,t)&=\int_0^t\int_{\mathbb{R}^2} e^{is\Phi(\xi,\eta)} \widehat{g_0}(\xi-\eta)\widehat{g_0}(\eta) \,d\eta ds \\
    &=2^{2l}\cdot\int_0^{t\cdot 2^{-4l}}\int_{\mathbb{R}^2} e^{is^\prime\widetilde{\Phi}(\xi^\prime,\eta^\prime)} \widehat{g_0}(2^{-3l}\xi^\prime-2^{-l}\eta^\prime)\widehat{g_0}(2^{-l}\eta^\prime) \,d\eta^\prime ds^\prime\\
    &=2^{2l}\cdot\int_0^{t\cdot 2^{-4l}}\int_{\left|\eta^\prime\right|\le (s^\prime)^{-\frac{1}{3}-0.01}} e^{is^\prime\widetilde{\Phi}(\xi^\prime,\eta^\prime)} \widehat{g_0}(2^{-3l}\xi^\prime-2^{-l}\eta^\prime)\widehat{g_0}(2^{-l}\eta^\prime) \,d\eta^\prime ds^\prime\\
    &\ \ \ \ \ +2^{2l}\cdot\int_0^{t\cdot 2^{-4l}}\int_{(s^\prime)^{-\frac{1}{3}-0.01}\le \left|\eta^\prime\right|\le 2^l} e^{is^\prime\widetilde{\Phi}(\xi^\prime,\eta^\prime)} \widehat{g_0}(2^{-3l}\xi^\prime-2^{-l}\eta^\prime)\widehat{g_0}(2^{-l}\eta^\prime) \,d\eta^\prime ds^\prime\\
    &\triangleq I_{11}+I_{12}.
\end{align*}
By volume estimate, we see that
\begin{align*}
    \left|I_{11}\right|\lesssim 2^{2l}\cdot \int_0^{t\cdot 2^{-4l}} \left(s^\prime\right)^{-\frac{2}{3}-0.02}\,ds^\prime \lesssim 2^{2l}\cdot t^{\frac{1}{3}}\cdot 2^{-\frac{4}{3}l}\approx 2^{\frac{2}{3}l}\cdot t^{\frac{1}{3}}\lesssim 2^{2l}.
\end{align*}
On the other hand, as for $I_{12}$, we let $s^\prime\sim 2^{-\lambda l}$ ($\lambda\ge 4-\frac{\log_2 t}{l}$) and write
$$I_{12}=2^{2l}\cdot \int_0^{t\cdot 2^{-4l}}\sum_{\frac{\lambda}{3}+0.01\lambda<\Tilde{\lambda}\le 1 \atop \Tilde{\lambda}l\in\mathbb{Z}} \int_{\left|\eta^\prime\right|\sim 2^{\Tilde{\lambda}l}} e^{is^\prime \Tilde{\Phi}(\xi^\prime,\eta^\prime)}\,d\eta^\prime.$$
Then, proceed as before in (\ref{3.4}) by integrating by parts in $\eta^\prime$ and we can conclude 
$$\left|I_{12}\right|\lesssim 2^{2l}\cdot \int_0^{t\cdot 2^{-4l}}\left(s^\prime\right)^{10}\,ds^\prime\lesssim 2^{2l}\cdot \int_0^{t\cdot 2^{-4l}}\left(s^\prime\right)^{-\frac{2}{3}-0.02}\,ds^\prime\lesssim 2^{\frac{2}{3}l}\cdot t^{\frac{1}{3}}\lesssim 2^{2l}.$$
To sum up, we get that
$$\left\|\widehat{P_{-3l}f_4}(\xi,t)\right\|_{L^\infty}\lesssim 2^{2l}.$$\par
\noindent $3^\circ$: $t \le 2^{l}$\par
In this case, we only need to do the most trivial volume estimate
$$\left|\widehat{P_{-3l}f_4}(\xi,t)\right|=\left|\int_0^t\int_{\mathbb{R}^2} e^{is\Phi(\xi,\eta)} \widehat{g_0}(\xi-\eta)\widehat{g_0}(\eta) \,d\eta ds\right|\lesssim t\lesssim 2^l\lesssim 2^{2l}.$$
\end{proof}
\vspace{2em}
Next, we still assume the frequency of $\widehat{f_4}$ (and $\widehat{f_5}$) is very low and we have to investigate the size of (Fourier-)space derivatives, since we have to integrate by parts in $\eta$ later on. 
\vspace{0.8em}
\begin{prop}
    For any $l\ge D\ge 0$, we fix $l$ and consider the function $\nabla_\xi\widehat{P_{-3l} f_4}(\xi,t)$. Then, we have that for any $t\ge 2^{4.05l}$ and $\left|\xi\right|\sim 2^{-3l}$, 
    $$\left|\nabla_\xi\widehat{P_{-3l} f_4}(\xi,t)\right|\approx 2^{5l}\approx 2^{3l}\cdot \left|\widehat{P_{-3l} f_4}(\xi,t)\right|.$$
    Moreover, for any $t\ge 2^{4.05l}$, $\left|\xi\right|\sim 2^{-3l}$ and multiindex $N$,
    $$\left|\partial^N_\xi\widehat{P_{-3l} f_4}(\xi,t)\right|\lesssim 2^{3\left|N\right|l}\cdot \left|\widehat{P_{-3l} f_4}(\xi,t)\right|,$$
    where $\partial_\xi^N=\partial_{\xi_1}^{N_1}\partial_{\xi_2}^{N_2}$ as usual.
\end{prop}
\begin{proof}
The proof is similar as the one of Proposition 3.1. In this proof, we again set $\xi^\prime=2^{3l}\xi$ and $t\sim 2^m\ge 2^{4.05l}$. \par
We first observe that
\begin{align*}
    \nabla_\xi\widehat{P_{-3l} f_4}(\xi,t)=&\int_0^{2^m}\int_{\mathbb{R}^2} \left[is(\nabla_\xi\Phi)(\xi,\eta)\right]e^{is\Phi(\xi,\eta)} \widehat{g_0}(\xi-\eta)\widehat{g_0}(\eta) \,d\eta ds\\
    &+\int_0^{2^m}\int_{\mathbb{R}^2} e^{is\Phi(\xi,\eta)} \nabla_\xi \widehat{g_0}(\xi-\eta)\widehat{g_0}(\eta) \,d\eta ds.\tag{3.7}\label{3.7}
\end{align*}
Since $\widehat{g_0}(\xi)$ is a Schwartz function, $\nabla_\xi \widehat{g_0}(\xi)$ is also a Schwartz function. Therefore, we can again apply the proof of Proposition 3.1 to conclude that 
\begin{align*}
    \int_0^{2^m}\int_{\mathbb{R}^2} e^{is\Phi(\xi,\eta)} \nabla_\xi \widehat{g_0}(\xi-\eta)\widehat{g_0}(\eta) \,d\eta ds\approx 2^{2l}.
\end{align*}
Now, we focus on the first term. In fact, we can write
\begin{align*}
    &\reallywidehat{\mathcal{B^*}}_{\left[3.99l,4.05l\right],\left[-1.03l,-0.99l\right]}\\
    =&\reallywidehat{\mathcal{B^*}}_{\left[3.99l,4.05l\right],\left[-1.03l,-0.99l\right]}(2^{-3l}\xi^\prime,t)=\reallywidehat{\mathcal{B^*}}_{\left[3.99l,4.05l\right],\left[-1.03l,-0.99l\right]}(\xi,t)\\
    \triangleq&\int_{2^{3.99l}}^{2^{4.05l}}\int_{\mathbb{R}^2} \left[is(\nabla_\xi\Phi)(\xi,\eta)\right]e^{is\Phi(\xi,\eta)} \widehat{g_0}(\xi-\eta)\widehat{g_0}(\eta)\mathbbm{1}_{\left[-1.03l,-0.99l\right]}(\eta)\,d\eta ds \\
    =&\int_{2^{3.99l}}^{2^{4.05l}}\int_{2^{-1.03l}\le\left|\eta\right|\le 2^{-0.99l}} \left[is(\nabla_\xi\Phi)(\xi,\eta)\right]e^{is\Phi(\xi,\eta)} \,d\eta ds \tag{3.8a}\label{3.8a} \\
    =&\int_{2^{3.99l}}^{2^{4.05l}}\int_{2^{-1.03l}\le\left|\eta\right|\le 2^{-0.99l}} \left[is(\nabla_\xi\Phi)(\xi,\eta)\right]e^{is\Phi(\xi,\eta)}\widehat{g_0}(\xi-\eta)\widehat{g_0}(\eta) \,d\eta ds \tag{3.8b}\label{3.8b}\\
\end{align*}
First, we consider the contribution of the above integral (\ref{3.8a}) when $\left|\eta\right|\notin\left[-1.03l,-0.99l\right]$. In fact, we have
\begin{align*}
    &\int_{2^{3.99l}}^{2^{4.05l}}\int_{\mathbb{R}^2} \left[is(\nabla_\xi\Phi)(\xi,\eta)\right]e^{is\Phi(\xi,\eta)} \,d\eta ds \\
    =&\int_{2^{3.99l}}^{2^{4.05l}}\int_{\left|\eta\right|\le 2^{-1.03l}} \left[is(\nabla_\xi\Phi)(\xi,\eta)\right]e^{is\Phi(\xi,\eta)} \,d\eta ds\\
    &\ \ \ \ +\int_{2^{3.99l}}^{2^{4.05l}}\int_{2^{-1.03l}\le \left|\eta\right|\le 2^{-0.99l}} \left[is(\nabla_\xi\Phi)(\xi,\eta)\right]e^{is\Phi(\xi,\eta)} \,d\eta ds\\
    &\ \ \ \ +\int_{2^{3.99l}}^{2^{4.05l}}\int_{\left|\eta\right|\ge 2^{-0.99l}} \left[is(\nabla_\xi\Phi)(\xi,\eta)\right]e^{is\Phi(\xi,\eta)} \,d\eta ds\triangleq I_1+I_2+I_3.
\end{align*}
As for $I_3$, we can decompose in terms of the length of $\eta$ 
\begin{align*}
    I_3&=\sum_{k_2>-0.99l}\int_{2^{3.99l}}^{2^{4.05l}}\int_{\left|\eta\right|\sim 2^{k_2}} \left[is(\nabla_\xi\Phi)(\xi,\eta)\right]e^{is\Phi(\xi,\eta)}\,d\eta ds
\end{align*}
Then, let $\eta^\prime=2^{\lambda l}\eta$ and $s^\prime=2^{-4l}s$ ($0\le\lambda \le 0.99$) and we see that
$$\widetilde{\Phi_\lambda}(\xi^\prime,\eta^\prime)=\frac{1}{2}\cdot2^{-(3-3\lambda)l}\left.\langle\xi^\prime,\eta^\prime\right.\rangle+O\left(2^{-12\lambda+4\lambda l}\left|\xi^\prime\right|^4+\left|\eta^\prime\right|^4\right),$$
which implies that
\begin{align*}
    \nabla_{\xi^\prime}\widetilde{\Phi_\lambda}(\xi^\prime,\eta^\prime)&=\frac{1}{2}\cdot 2^{-(3-3\lambda)l}\eta^\prime+O\left(2^{-12l+4\lambda l}\left|\xi^\prime\right|^3\right)\\
    &=2^{-(3-4\lambda)l}\left(\frac{1}{2}\cdot \eta+O\left(\left|\xi\right|^3\right)\right)=2^{-(3-4\lambda)l}\left(\nabla_\xi\Phi\right)(\xi,\eta).
\end{align*}
Therefore, we get
\begin{align*}
    I_3&=\sum_{k_2>-0.99l} 2^{4l-2\lambda l} \int_{2^{-0.01l}}^{2^{0.05l}}\int_{\left|\eta^\prime\right|\sim 1} \left[is^\prime 2^{(7-4\lambda)l} \nabla_\xi\widetilde{\Phi_\lambda}(\xi^\prime,\eta^\prime)\right]e^{i\left(s^\prime\cdot 2^{(4-4\lambda)l}\right)\widetilde{\Phi_\lambda}(\xi^\prime,\eta^\prime)} \,d\eta^\prime ds^\prime.\\
    &=\sum_{k_2>-0.99l} i\cdot 2^{11l-6\lambda l} \int_{2^{-0.01l}}^{2^{0.05l}}s^\prime \int_{\left|\eta^\prime\right|\sim 1} \left[\nabla_\xi\widetilde{\Phi_\lambda}(\xi^\prime,\eta^\prime)\right]e^{i\left(s^\prime\cdot 2^{(4-4\lambda)l}\right)\widetilde{\Phi_\lambda}(\xi^\prime,\eta^\prime)} \,d\eta^\prime ds^\prime.
\end{align*}
Note that $\left|\nabla_{\eta^\prime}\widehat{g_0}(2^{-3l}\xi^\prime-2^{-\lambda l}\eta^\prime)\right|\sim 2^{-\lambda l}$, $\left|\nabla_{\eta^\prime}\widehat{g_0}(2^{-\lambda l}\eta^\prime)\right|\sim 2^{-\lambda l}$ and $s^\prime\cdot 2^{(4-4\lambda)l}\ge 2^{0.03l}$. 
Lemma 2.1 tells us that $\nabla_{\eta^\prime}\widetilde{\Phi_\lambda}(\xi^\prime,\eta^\prime)\neq 0$, which implies that
\begin{align*}
    I_3&=\sum_{k_2>-0.99l} i\cdot 2^{11l-6\lambda l} \int_{2^{-0.01l}}^{2^{0.05l}} O\left(s^\prime\left(s^\prime\cdot 2^{(4-4\lambda)l}\right)^{-N}\right) ds^\prime
\end{align*}
So, by take $N,D$ large enough, we can get 
\begin{align*}
    \left|I_3\right|\le\sum_{k_2>-0.99l} A_N 2^{11l-6\lambda l-0.04Nl+0.1l}=\frac{4}{3} A_N 2^{(11.1-0.04N)l}\le 2^{-10l}
\end{align*}
On the other hand, as for $I_1$, let $\eta^\prime=2^l \eta$ and $s^\prime=2^{-4l}s$ and we get
\begin{align*}
    I_1=2^{2l} \int_{2^{-0.01l}}^{2^{0.05l}} \int_{\left|\eta^\prime\right|\le 2^{-0.03l}}2^{3l}\left[is^\prime(\nabla_\xi\widetilde{\Phi})(\xi^\prime,\eta^\prime)\right]e^{is^\prime\widetilde{\Phi}(\xi^\prime,\eta^\prime)}\,d\eta^\prime ds^\prime.
\end{align*}
So, by volume estimate, we get 
\begin{align*}
    \left|I_1\right|\le 2^{2l}\cdot 2^{3l}\cdot 2^{0.05l}\cdot C_4\cdot\pi\cdot 2^{-0.06l}=C_4\cdot \pi \cdot2^{4.99l}.
\end{align*}
Thus, in view of (\ref{3.8a}), we have
\begin{align*}
    \reallywidehat{\mathcal{B^*}}_{\left[3.99l,4.05l\right],\left[-1.03l,-0.99l\right]}=\int_{2^{3.99l}}^{2^{4.05l}} \int_{\mathbb{R}^2} \left[is(\nabla_\xi\Phi)(\xi,\eta)\right]e^{is\Phi(\xi,\eta)}\,d\eta ds+O\left(2^{4.99l}\right)
\end{align*}
Let $\eta^\prime=2^l \eta$ and $s^\prime=2^{-4l} s$. Then, we get
\begin{align*}
    \reallywidehat{\mathcal{B^*}}_{\left[3.99l,4.05l\right],\left[-1.03l,-0.99l\right]}&=2^{2l}\cdot 2^{3l}\cdot\int_{2^{-0.01l}}^{2^{0.05l}} \int_{\mathbb{R}^2} \left[is^\prime(\nabla_\xi\widetilde{\Phi})(\xi^\prime,\eta^\prime)\right]e^{is^\prime\widetilde{\Phi}(\xi^\prime,\eta^\prime)} \,d\eta^\prime ds^\prime+O\left(2^{4.99l}\right) \\
    &\triangleq 2^{5l}\cdot\int_{2^{-0.01l}}^{2^{0.05l}}H(s^\prime)\,ds^\prime+O\left(2^{4.99l}\right).\tag{3.9}\label{3.9}
\end{align*}
Now, if $\left|s^\prime\right|\le 2^{-0.01l}<1$, then we decompose
\begin{align*}
    H(s^\prime)&=i\cdot\int_{\mathbb{R}^2} \left[s^\prime(\nabla_\xi\widetilde{\Phi})(\xi^\prime,\eta^\prime)\right]e^{is^\prime\widetilde{\Phi}(\xi^\prime,\eta^\prime)} \,d\eta^\prime\\
    &=i\cdot\int_{\left|\eta^\prime\right|\le\left(s^\prime\right)^{-\frac{1}{3}-0.01}} \left[s^\prime(\nabla_\xi\widetilde{\Phi})(\xi^\prime,\eta^\prime)\right]e^{is^\prime\widetilde{\Phi}(\xi^\prime,\eta^\prime)} \,d\eta^\prime\\
    &\ \ \ +i\cdot\int_{\left|\eta^\prime\right|\ge\left(s^\prime\right)^{-\frac{1}{3}-0.01}} \left[s^\prime(\nabla_\xi\widetilde{\Phi})(\xi^\prime,\eta^\prime)\right]e^{is^\prime\widetilde{\Phi}(\xi^\prime,\eta^\prime)} \,d\eta^\prime\\
    &\triangleq I_{11}+I_{12}.
\end{align*}
By volume estimate, we see that
\begin{align*}
    \left|I_{11}\right|\le C_4\cdot \pi\cdot\left(s^\prime\right)^{-\frac{2}{3}-0.02}.\tag{3.10}\label{3.10}
\end{align*}
In terms of $I_{12}$, we let $s^\prime\sim 2^{-\lambda l}$ ($\lambda\ge 0.01$) and write 
$$I_{12}=\sum_{\frac{\lambda}{3}+0.01\lambda<\Tilde{\lambda}\le 1 \atop \Tilde{\lambda}l\in\mathbb{Z}} i\cdot\int_{\left|\eta^\prime\right|\sim 2^{\Tilde{\lambda}l}} \left[s^\prime(\nabla_\xi\widetilde{\Phi})(\xi^\prime,\eta^\prime)\right]e^{is^\prime \Tilde{\Phi}(\xi^\prime,\eta^\prime)}\,d\eta^\prime.$$
Then, note that $\nabla_{\eta^\prime}\left[s^\prime \Tilde{\Phi}(\xi^\prime,\eta^\prime)\right]=s^\prime\left(\frac{1}{2}\xi^\prime+O(\left|\eta^\prime\right|^2\eta^\prime)\right)$, which implies that $\left|\nabla_{\eta^\prime}\left[s^\prime \Tilde{\Phi}(\xi^\prime,\eta^\prime)\right]\right|\ge 2^{-\lambda l}\cdot 2^{3\Tilde{\lambda}l}\ge 2^{0.03\lambda l}$. Also notice that $\left|\nabla_{\eta^\prime}\widehat{g_0}(2^{-3l}\xi^\prime-2^{-l}\eta^\prime)\right|\sim 2^{-l}$, $\left|\nabla_{\eta^\prime}\widehat{g_0}(2^{-l}\eta^\prime)\right|\sim 2^{-l}$ and $\left|\nabla_{\eta^\prime}\left[s^\prime(\nabla_\xi\widetilde{\Phi})(\xi^\prime,\eta^\prime))\right]\right|\lesssim 1$. Thus, we can keep doing integration by parts in $\eta^\prime$ such that for all $\frac{\lambda}{3}+0.01\lambda<\Tilde{\lambda}<1$, we have
$$\left|\int_{\left|\eta^\prime\right|\sim 2^{\Tilde{\lambda}l}} \left[s^\prime(\nabla_\xi\widetilde{\Phi})(\xi^\prime,\eta^\prime)\right]e^{is^\prime \Tilde{\Phi}(\xi^\prime,\eta^\prime)}\,d\eta^\prime\right|\le 2^{-11\lambda l}.$$
This gives
\begin{align*}
    \left|I_{12}\right|\le \sum_{\frac{\lambda}{3}+0.01\lambda<\Tilde{\lambda}\le 1 \atop \Tilde{\lambda}l\in\mathbb{Z}} 2^{-11\lambda l} \le \left|l\right|\cdot 2^{-11\lambda l}\le \frac{1}{\lambda}\cdot\left|\lambda l\right|\cdot 2^{-11\lambda l}\le \frac{1}{\lambda}\cdot 2^{-10\lambda l}\le 100\cdot \left(s^\prime\right)^{10}\le \pi \cdot \left(s^\prime\right)^{-\frac{2}{3}-0.02}.\tag{3.11}\label{3.11}
\end{align*}
Combining (\ref{3.10}) and (\ref{3.11}), we conclude
\begin{align*}
    \left|\int_0^{2^{-0.01l}} H(s^\prime)\,ds^\prime\right|\le C_4\cdot \pi\cdot\int_0^{2^{-0.01l}} (s^\prime)^{-\frac{2}{3}-0.02}\,ds^\prime=C_4\cdot \frac{150\pi}{47}\cdot 2^{-\frac{47}{15000}l}
\end{align*}
On the other hand, we suppose that $\left|s^\prime\right|\ge 2^{0.05l}$. Note that $\widehat{g_0}$ has a compact support. If there exists $\eta^\prime(\xi^\prime)$ such that $\nabla_{\eta^\prime}\Tilde{\Phi}(\xi^\prime,\eta^\prime(\xi^\prime))$, then we can write
\begin{align*}
    H(s^\prime)=C e^{is^\prime \Tilde{\Phi}(\xi^\prime,\eta^\prime(\xi^\prime))}\cdot\frac{1}{s^\prime}\cdot\left[s^\prime(\nabla_\xi\widetilde{\Phi})(\xi^\prime,\eta^\prime(\xi^\prime))\right]+O\left(\frac{1}{(s^\prime)^2}\right).
\end{align*}
Therefore,
\begin{align*}
    \left|\int_{2^{0.05l}}^{+\infty} H(s^\prime)\,ds^\prime\right|&\lesssim \left|\int_{2^{0.05l}}^{+\infty} e^{is^\prime\Tilde{\Phi}(\xi^\prime,\eta^\prime(\xi^\prime))}\,ds^\prime\right|+\int_{2^{0.05l}}^{+\infty} \frac{ds^\prime}{(s^\prime)^2}\\
    &\lesssim \sum_{m\ge 0.05l}\int e^{is^\prime\Tilde{\Phi}(\xi^\prime,\eta^\prime(\xi^\prime))}\varphi\left(\frac{s^\prime}{2^m}\right)\,ds^\prime+ 2^{-0.05l}\\
    &\lesssim \sum_{m\ge 0.05l}2^{2m}\hat{\varphi}(2^m\Tilde{\Phi}(\xi^\prime,\eta^\prime(\xi^\prime)))+2^{-0.05l}\\
    &\lesssim \sum_{m\ge 0.05l} 2^{-m}+2^{-0.05l}\lesssim  2^{-0.05l}.\tag{3.12}\label{3.12}
\end{align*}
Finally, in view of (\ref{3.9}), we see that
\begin{align*}
    \reallywidehat{\mathcal{B^*}}_{\left[3.99l,4.05l\right],\left[-1.03l,-0.99l\right]}=2^{5l}\cdot\int_0^{+\infty}H(s^\prime)\,ds^\prime+O\left(2^{4.997l}\right).\tag{3.13}\label{3.13}
\end{align*}
Note that
\begin{align*}
    \int_0^{+\infty}H(s^\prime)\,ds^\prime&=\int_0^{+\infty}\int_{\mathbb{R}^2} e^{is^\prime\widetilde{\Phi}(\xi^\prime,\eta^\prime)} \,d\eta^\prime ds^\prime\neq 0
\end{align*}
is a nonzero constant independent of $l$.\par
By a similar argument as in Proposition 3.1, we also have
\begin{align*}
    & \int_0^{+\infty} H(s^\prime)\,ds^\prime\\
    \approx&\int_0^{+\infty}\int_{\mathbb{R}^2} \left[is^\prime(\nabla_\xi\widetilde{\Phi})(\xi^\prime,\eta^\prime)\right]e^{is^\prime\widetilde{\Phi}(\xi^\prime,\eta^\prime)} \widehat{g_0}(2^{-3l}\xi^\prime-2^{-l}\eta^\prime)\widehat{g_0}(2^{-l}\eta^\prime) \,d\eta^\prime ds^\prime\\
    \approx&\int_0^{2^{m-4l}}\int_{\mathbb{R}^2} \left[is^\prime(\nabla_\xi\widetilde{\Phi})(\xi^\prime,\eta^\prime)\right]e^{is^\prime\widetilde{\Phi}(\xi^\prime,\eta^\prime)} \widehat{g_0}(2^{-3l}\xi^\prime-2^{-l}\eta^\prime)\widehat{g_0}(2^{-l}\eta^\prime) \,d\eta^\prime ds^\prime.
\end{align*}
Then we can compute
\begin{align*}
    \nabla_\xi\widehat{P_{-3l}f_4}(\xi,t)&=\int_0^{2^m}\int_{\mathbb{R}^2} \left[is(\nabla_\xi\Phi)(\xi,\eta)\right]e^{is\Phi(\xi,\eta)} \widehat{g_0}(\xi-\eta)\widehat{g_0}(\eta) \,d\eta ds\\
    &+\int_0^{2^m}\int_{\mathbb{R}^2} e^{is\Phi(\xi,\eta)} \nabla_\xi \widehat{g_0}(\xi-\eta)\widehat{g_0}(\eta) \,d\eta ds.\\
    &=2^{5l}\cdot\int_0^{2^{m-4l}}\int_{\mathbb{R}^2} \left[is^\prime(\nabla_\xi\widetilde{\Phi})(\xi^\prime,\eta^\prime)\right]e^{is^\prime\widetilde{\Phi}(\xi^\prime,\eta^\prime)} \widehat{g_0}(2^{-3l}\xi^\prime-2^{-l}\eta^\prime)\widehat{g_0}(2^{-l}\eta^\prime) \,d\eta^\prime ds^\prime\\
    &\ \ \ \ +O\left(2^{2l}\right)\\
    &\approx 2^{5l}\cdot \int_0^{+\infty} H(s^\prime)\,ds^\prime+O\left(2^{2l}\right) \approx 2^{5l}.
\end{align*}
Namely, $\left|\nabla_\xi\widehat{P_{-3l}f_4}(\xi,t)\right|\approx 2^{5l}$, provided that $m\ge 4.05l$ and $\left|\xi\right|\sim 2^{-3l}$. Moreover, we have that
\begin{align*}
   \nabla_\xi\widehat{P_{-3l}f_4}(\xi,t)=\reallywidehat{\mathcal{B^*}}_{\left[3.99l,4.05l\right],\left[-1.03l,-0.99l\right]}+O\left(2^{4.997l}\right).
\end{align*} \par
Next, we consider higher derivatives. In fact, in view of (\ref{3.7}), we only need focus on the first term, since the first term is the main term. The key observations here are 
$$\left|\partial^N_\xi\Phi\right|\lesssim 2^{-3l}\ll 2^{-l},\ \ \ \ \ \mbox{if }\left|N\right|\ge 2,$$
and
$$\partial_\xi^N e^{is\Phi(\xi,\eta)}=e^{is\Phi(\xi,\eta)}\sum_{r=1}^{\left|N\right|} \sum_{\alpha^{(1)}+\dots+\alpha^{(r)}=\left|N\right|\atop \sum_{i=1}^r \alpha_i^{(j)}=N_j\ \ (j=1,2)} C\left(N,r;\alpha^{(1)},\dots,\alpha^{(r)}\right)\cdot(is)^r\cdot\partial_\xi^{\alpha^{(1)}}\Phi\cdot\partial_\xi^{\alpha^{(2)}}\Phi\dots\partial_\xi^{\alpha^{(r)}}\Phi.$$
Then, we just need to do the proof analogously as before and the higher derivative result follows from the fact that when $s\sim 2^{4l}$ and $\left|\eta\right|\sim 2^{-l}$, we have
\begin{align*}
    \left|(is)^r\cdot\partial_\xi^{\alpha^{(1)}}\Phi\cdot\partial_\xi^{\alpha^{(2)}}\Phi\dots\partial_\xi^{\alpha^{(r)}}\Phi\right|\lesssim 2^{4rl}\cdot 2^{-rl}=2^{3rl}=\left(2^{3l}\right)^r.
\end{align*}
\end{proof}
\vspace{2em}
Next, we consider the middle frequency case of $\widehat{f_4}$ (and $\widehat{f_5}$).
\vspace{0.8em}
\begin{prop}
    Fix $D\gg 1$. If $-D\le k\le 0$, then for any $t\ge 0$ we have the following\par
    (a) $\left\|\widehat{P_k f_4}(\xi,t)\right\|_{L^\infty_\xi}\lesssim_D 1$.\par
    (b) $\left\|\nabla_\xi\widehat{P_k f_4}(\xi,t)\right\|_{L^\infty_\xi}\lesssim_D 1$.\par
    (c) $\left\|\partial_\xi^N\widehat{P_k f_4}(\xi,t)\right\|_{L^\infty_\xi}\lesssim_{N,D} 1$, for all multiindex $N$.
\end{prop}
\begin{proof}
(a) If $t\le \max\left(2^{-1.5k},2^{100}\right)$, then we only need to use volume estimate to get
\begin{align*}
    \left|\widehat{P_k f_4}(\xi,t)\right|&=\left|\int_0^t \int_{\mathbb{R}^2} e^{is\Phi(\xi,\eta)}\widehat{g_0}(\xi-\eta)\widehat{g_0}(\eta)\,d\eta ds\right|\\
    &\le \left|t\right|\cdot 1.2^2\le 2^{1.6D}.
\end{align*}
On the other hand, if $t\ge \max\left(2^{-1.5k},2^{100}\right)$, then we write
\begin{align*}
    \widehat{P_k f_4}(\xi,t)&=\int_0^{\max\left(2^{-1.5k},2^{100}\right)} \int_{\mathbb{R}^2} e^{is\Phi(\xi,\eta)}\widehat{g_0}(\xi-\eta)\widehat{g_0}(\eta)\,d\eta ds\\
    &\ \ \ \ +\int_{\max\left(2^{-1.5k},2^{100}\right)}^t \int_{\mathbb{R}^2} e^{is\Phi(\xi,\eta)}\widehat{g_0}(\xi-\eta)\widehat{g_0}(\eta)\,d\eta ds\triangleq I_1+I_2
\end{align*}
By volume estimate, we again see that $\left|I_1\right|\le 2^{1.6D}$. Next, we consider the term $I_2$ and suppose $\max\left(2^{-1.5k},2^{100}\right)\le s\le t$. We now observe that at least one of $\xi-\eta$ and $\eta$ is greater than $2^{k-2}$, which implies that 
$$\left|s\cdot\nabla_\eta\Phi\right|\sim\left|s\right|\cdot\left(\left|\xi\right|+O\left(\left|\eta\right|^3\right)\right)\sim\left|s\right|\cdot\max\left(\left|\xi\right|,\left|\eta\right|^3\right) \ge \left|s\right|\cdot\left|\xi\right|\ge \left|s\right|^{\frac{1}{3}}.$$
Then, integration by parts in $\eta$ gives that
\begin{align*}
    \left|I_2\right|\lesssim\int_{\max\left(2^{-1.5k},2^{100}\right)}^t \frac{ds}{s^2}\lesssim 2^{-100}.
\end{align*}
To sum up, we've proved that $\left\|\widehat{P_k f_4}(\xi,t)\right\|_{L^\infty_\xi}\lesssim_D 1$.\par
(b) Part (b) can be proven similarly as part (a). If $t\le \max\left(2^{-1.5k},2^{100}\right)$, then we only need to use volume estimate to get
\begin{align*}
    \left|\nabla_\xi\widehat{P_k f_4}(\xi,t)\right|&\le\left|\int_0^t \int_{\mathbb{R}^2} \left[is\nabla_\xi\Phi(\xi,\eta)\right]e^{is\Phi(\xi,\eta)}\widehat{g_0}(\xi-\eta)\widehat{g_0}(\eta)\,d\eta ds\right|\\
    +&\left|\int_0^t \int_{\mathbb{R}^2}e^{is\Phi(\xi,\eta)} \nabla_\xi\widehat{g_0}(\xi-\eta)\widehat{g_0}(\eta)\,d\eta ds\right|\\
    &\lesssim \frac{1}{2}\left|t\right|^2\cdot 1.2^2+\left|t\right|\cdot 1.2^2\lesssim 2^{3D}.
\end{align*}
On the other hand, if $t\ge \max\left(2^{-1.5k},2^{100}\right)$, then we can similarly write
\begin{align*}
    \nabla_\xi\widehat{P_k f_4}(\xi,t)&=\int_0^{\max\left(2^{-1.5k},2^{100}\right)} \int_{\mathbb{R}^2} \left[is\nabla_\xi\Phi(\xi,\eta)\right]e^{is\Phi(\xi,\eta)}\widehat{g_0}(\xi-\eta)\widehat{g_0}(\eta)\,d\eta ds\\
    &\ \ \ \ +\int_{\max\left(2^{-1.5k},2^{100}\right)}^t \int_{\mathbb{R}^2} \left[is\nabla_\xi\Phi(\xi,\eta)\right]e^{is\Phi(\xi,\eta)}\widehat{g_0}(\xi-\eta)\widehat{g_0}(\eta)\,d\eta ds\\
    &\ \ \ \ +\int_0^t \int_{\mathbb{R}^2}e^{is\Phi(\xi,\eta)} \nabla_\xi\widehat{g_0}(\xi-\eta)\widehat{g_0}(\eta)\,d\eta ds\triangleq I_1+I_2+I_3,
\end{align*}
where by volume estimate as before, we see that $\left|I_1\right|, \left|I_3\right|\lesssim 2^{2.6D}\lesssim_D 1$. When it comes to the term $I_2$, we suppose $\max\left(2^{-1.5k},2^{100}\right)\le s\le t$ and again observe that $\left|s\cdot\nabla_\eta\Phi\right|\gtrsim \left|s\right|^{\frac{1}{3}}$. Also note that $\left|\nabla_\eta\Phi\right|\lesssim 1$. Then, integration by parts in $\eta$ gives that
\begin{align*}
    \left|\int_{\mathbb{R}^2}\nabla_\xi\Phi(\xi,\eta) e^{is\Phi(\xi,\eta)}\widehat{g_0}(\xi-\eta)\widehat{g_0}(\eta)\,d\eta ds\right|\lesssim \frac{1}{s^3}.
\end{align*}
Therefore, we see that
\begin{align*}
    \left|I_2\right|\lesssim \int_{\max\left(2^{-1.5k},2^{100}\right)}^t s\cdot \frac{ds}{s^3}\lesssim 2^{-100}.
\end{align*}
To sum up, we've proved that $\left\|\nabla_\xi\widehat{P_k f_4}(\xi,t)\right\|_{L^\infty_\xi}\lesssim_D 1$.\par
(c) Higher derivatives case can be done analogously. 
\end{proof}
\vspace{0.8em}
\begin{remark}
\ \par
(a) The Duhamel's formula of $\widehat{P_k f_4}(\xi,t)$ also tells us that if $k\ge 2$, then $\widehat{P_k f_4}(\xi,t)=0$ for all $t$.\par
(b) By reviewing the proof of Proposition 3.3, we can also see that if $-D\le k\le 0$ ($D\gg 1$), then for any $t\ge 0$
$$\left\|\nabla_\xi\widehat{P_k f_4}(\xi,t)\right\|_{L^\infty_\xi}\lesssim_D \left\|\widehat{P_k f_4}(\xi,t)\right\|_{L^\infty_\xi},$$
and 
$$\left\|\partial_\xi^N\widehat{P_k f_4}(\xi,t)\right\|_{L^\infty_\xi}\lesssim_{N,D} \left\|\widehat{P_k f_4}(\xi,t)\right\|_{L^\infty_\xi}.$$
\end{remark}
\vspace{2em}
It turns out that we also need to integrate by parts in time variable later on. Therefore, we have to figure out the behaviour of the time derivative of $\widehat{f_4}$ (and $\widehat{f_5}$) as well.
\vspace{0.8em}
\begin{prop}
    For any $l\ge D\ge 0$, we fix $l$ and consider the function $\partial_t \widehat{P_{-3l} f_4}(\xi,t)$. Then, we have that for any $t\ge 2^{4.05l}$ and $\left|\xi\right|\sim 2^{-3l}$,\par 
    $\left|\partial_t \widehat{P_{-3l} f_4}(\xi,t)\right|\approx 2^{2l}\cdot\frac{e^{i(\alpha\cdot 2^{-4l}\cdot t)}}{t}$, where $\alpha\triangleq\widetilde{\Phi}(\xi^\prime,\eta^\prime(\xi^\prime))\neq 0$.
\end{prop}
\begin{proof}
\ \par
As before, we let $s^\prime=2^{-4l}\cdot s$, $\xi^\prime=2^{3l}\cdot \xi$ and $\eta^\prime=2^l\cdot \eta$. Then, by Duhamel's formula and the method of stationary phase, we write
\begin{align*}
    \widehat{P_{-3l} f_4}(\xi,t)&=\widehat{P_{-3l} f_4}(2^{-3l}\xi^\prime,t)\\
    &=\int_0^t \int_{\mathbb{R}^2} e^{is\Phi(\xi,\eta)}\widehat{g_0}(\xi-\eta)\widehat{g_0}(\eta)\,d\eta ds\\
    &=2^{2l}\cdot \int_0^{t\cdot 2^{-4l}}\int_{\mathbb{R}^2} e^{is^\prime \widetilde{\Phi}(\xi^\prime,\eta^\prime)}\widehat{g_0}(2^{-3l}\xi^\prime-2^{-l}\eta^\prime)\widehat{g_0}(2^{-l}\eta^\prime)\,d\eta^\prime ds^\prime\\
    &=2^{2l}\cdot \int_0^{t\cdot 2^{-4l}} \left[e^{is^\prime \widetilde{\Phi}(\xi^\prime,\eta^\prime(\xi^\prime))}\cdot \frac{1}{s^\prime}\cdot \widehat{g_0}(2^{-3l}\xi^\prime-2^{-l}\eta^\prime(\xi^\prime))\widehat{g_0}(2^{-l}\eta^\prime(\xi^\prime))+O\left(\frac{1}{(s^\prime)^2}\right)\right]\,ds^\prime\\
    &=2^{2l}\cdot \int_0^{t\cdot 2^{-4l}} \frac{e^{i\alpha s^\prime}}{s^\prime}\,ds^\prime+C\cdot 2^{2l}\cdot \int_0^{t\cdot 2^{-4l}} \frac{ds^\prime}{(s^\prime)^2}\\
    &=2^{2l}\cdot \int_0^{\alpha\cdot t\cdot 2^{-4l}} \frac{e^{i \widetilde{s^\prime}}}{\widetilde{s^\prime}}\,d\widetilde{s^\prime}+C\cdot 2^{2l}\cdot \int_0^{t\cdot 2^{-4l}} \frac{ds^\prime}{(s^\prime)^2},
\end{align*}
where in view of Corollary 2.3, $\eta^\prime(\xi^\prime)$ is the vector such that $\nabla_{\eta^\prime}\widetilde{\Phi}(\xi^\prime,\eta^\prime(\xi^\prime))=0$ and $\widetilde{\Phi}(\xi^\prime,\eta^\prime(\xi^\prime))\triangleq\alpha\neq 0$. Denote $\displaystyle{\mbox{Ci}(x)\triangleq\int_0^x \frac{\cos u}{u}\,du}$ and $\displaystyle{\mbox{Si}(x)\triangleq\int_0^x \frac{\sin u}{u}\,du}$. Now, we can write
\begin{align*}
    \widehat{P_{-3l} f_4}(\xi,t)=2^{2l}\left(\mbox{Ci}(\alpha \,t\, 2^{-4l})+i\cdot\mbox{Si}(\alpha\, t\, 2^{-4l})\right)+C\cdot 2^{2l}\cdot \int_0^{t\cdot 2^{-4l}} \frac{ds^\prime}{(s^\prime)^2}.
\end{align*}
Take the derivative of $t$ and we get
\begin{align*}
    \partial_t \widehat{P_{-3l} f_4}(\xi,t)&=2^{2l}\cdot \alpha \cdot 2^{-4l}\left(\mbox{Ci}^\prime(\alpha \,t\, 2^{-4l})+i\cdot\mbox{Si}^\prime(\alpha\, t\, 2^{-4l})\right)+C\cdot 2^{2l}\cdot \frac{2^{4l}}{t^2} \\
    &=2^{2l}\cdot \alpha \cdot 2^{-4l}\left(\frac{\cos\left(\alpha\,t\,2^{-4l}\right)}{\alpha\,t\,2^{-4l}}+i\cdot\frac{\cos\left(\alpha\,t\,2^{-4l}\right)}{\alpha\,t\,2^{-4l}}\right)+C\cdot 2^{2l}\cdot \frac{2^{4l}}{t^2}\\
    &=2^{2l}\cdot\frac{e^{i\left(\alpha\cdot 2^{-4l}\cdot t\right)}}{t}+C\cdot 2^{2l}\cdot \frac{2^{4l}}{t^2}.
\end{align*}
Since we've assumed $t\ge 2^{4.05l}$, we see that $\left|C\cdot 2^{2l}\cdot \frac{2^{4l}}{t^2}\right|\lesssim 2^{2l}\cdot\displaystyle{\frac{1}{t}}$. This implies that 
$$\left|\partial_t \widehat{P_{-3l} f_4}(\xi,t)\right|\approx 2^{2l}\cdot\frac{e^{i(\alpha\cdot 2^{-4l}\cdot t)}}{t},$$
provided that $\left|\xi\right|\sim 2^{-3l}$.
\end{proof}
\vspace{0.8em}
\begin{remark}
Proposition 3.1 - 3.3, Remark 3.4 and Proposition 3.5 also hold for the function $f_5$, since
$$\sqrt{2\left|\xi\right|^2+4}-\sqrt{4\left|\xi-\eta\right|^2+16}+\sqrt{2\left|\eta\right|^2+4}=\frac{1}{2}\left\langle \xi,\eta\right\rangle+O\left(\left|\xi\right|^4+\left|\eta\right|^4\right).$$
\end{remark}
\vspace{2em}
Now, having found all properties of the first iterative outputs $\widehat{f_4}$ and $\widehat{f_5}$, it's time for us to do the second iteration. This time, we will first find the pointwise size of $\widehat{P_{-9l}f_6}$ with pointing out the main contribution at $\left|\xi\right|\sim 2^{-9l}$, $\left|\eta\right|\sim 2^{-3l}$ and $t\sim 2^{12l}$. Then, we will compute the $L^2$ norm of $\widehat{P_{-9l}f_6}$, which will be a very large number.
\vspace{0.8em}
\begin{prop}
    For any $l\ge 10D\gg 1$, we fix $l$ and consider the function $\widehat{P_{-9l}f_6}(\xi,t)$. Then, we have that for any $t\ge 2^{12.05l}$ and $\left|\xi\right|\sim 2^{-9l}$, $\left|\widehat{P_{-9l}f_6}(\xi,t)\right|\approx 2^{10l}$. Moreover, we can write 
    $$\widehat{P_{-9l}f_6}(\xi,t)\approx\reallywidehat{\mathcal{B}}_{\left[11.99l,12.05l\right],\left[-3.09l,-2.99l\right]}+e(\xi,t)\approx 2^{10l}+e(\xi,t),$$
    where $\left|e(\xi,t)\right|\lesssim 2^{9.997l}$.
\end{prop}
\begin{proof}
\ \par
The proof is similar as in Proposition 3.1, except that we now have non-Schwartz inputs. In this proof, we set $\xi^{\prime\prime}=2^{9l}\xi$ and $t\sim 2^m$. Also, we correspondingly use a different rescaling of the phase function $\Phi$, namely let $\Phi(\xi,\eta)=2^{-12l}\widetilde{\Phi}(\xi^{\prime\prime},\eta^{\prime\prime})$ ($\eta^{\prime\prime}=2^{3l}\eta$) and more generally $\Phi(\xi,\eta)=2^{-4\lambda l}\widetilde{\Phi}(\xi^{\prime\prime},\eta^{\prime\prime})$ ($\eta^{\prime\prime}=2^{\lambda l}\eta$, $0\le \lambda\le 3$). \par
The main observation in our proof is the main term of the Duhamel's formula of $\widehat{P_{-9l}f_6}(\xi,t)$ is $\reallywidehat{\mathcal{B}}_{\left[11.99l,12.05l\right],\left[-3.09l,-2.99l\right]}$.\par
In fact, we can write
\begin{align*}
    &\reallywidehat{\mathcal{B}}_{\left[11.99l,12.05l\right],\left[-3.09l,-2.99l\right]}\\
    =&\reallywidehat{\mathcal{B}}_{\left[11.99l,12.05l\right],\left[-3.09l,-2.99l\right]}(2^{-9l}\xi^{\prime\prime},t)=\reallywidehat{\mathcal{B}}_{\left[11.99l,12.05l\right],\left[-3.09l,-2.99l\right]}(\xi,t) \\
    \triangleq&\int_{2^{11.99l}}^{2^{12.05l}}\int_{\mathbb{R}^2} e^{is\Phi(\xi,\eta)} \widehat{f_4}(\xi-\eta,s)\widehat{f_5}(\eta,s)\mathbbm{1}_{\left[-3.09l,-2.99l\right]}(\eta)\,d\eta ds \\
    =&\int_{2^{11.99l}}^{2^{12.05l}}\int_{2^{-3.09l}\le\left|\eta\right|\le 2^{-2.99l}} e^{is\Phi(\xi,\eta)}\widehat{f_4}(\xi-\eta,s)\widehat{f_5}(\eta,s) \,d\eta ds \tag{3.14}\label{3.14}\\
\end{align*}
Now, we consider the contribution of the above integral (\ref{3.14}) when $\left|\eta\right|\notin\left[-3.09l,-2.99l\right]$. In fact, we have
\begin{align*}
    &\int_{2^{11.99l}}^{2^{12.05l}}\int_{\mathbb{R}^2} e^{is\Phi(\xi,\eta)}\widehat{f_4}(\xi-\eta,s)\widehat{f_5}(\eta,s) \,d\eta ds \\
    =&\int_{2^{11.99l}}^{2^{12.05l}}\int_{\left|\eta\right|\le 2^{-3.09l}} e^{is\Phi(\xi,\eta)}\widehat{f_4}(\xi-\eta,s)\widehat{f_5}(\eta,s) \,d\eta ds\\
    &\ \ \ \ +\int_{2^{11.99l}}^{2^{12.05l}}\int_{2^{-3.09l}\le \left|\eta\right|\le 2^{-2.99l}} e^{is\Phi(\xi,\eta)}\widehat{f_4}(\xi-\eta,s)\widehat{f_5}(\eta,s) \,d\eta ds\\
    &\ \ \ \ +\int_{2^{11.99l}}^{2^{12.05l}}\int_{\left|\eta\right|\ge 2^{-2.99l}} e^{is\Phi(\xi,\eta)}\widehat{f_4}(\xi-\eta,s)\widehat{f_5}(\eta,s) \,d\eta ds\triangleq I_1+I_2+I_3.
\end{align*}
As for $I_3$, we can decompose in terms of the length of $\eta$ 
$$I_3=\sum_{k_2>-2.99l}\int_{2^{11.99l}}^{2^{12.05l}}\int_{\left|\eta\right|\sim 2^{k_2}} e^{is\Phi(\xi,\eta)}\widehat{f_4}(\xi-\eta,s)\widehat{f_5}(\eta,s)\,d\eta ds.$$
Then, let $\eta^{\prime\prime}=2^{\lambda l}\eta$ and $s^{\prime\prime}=2^{-12l}s$ ($0\le\lambda \le 2.99$) and get
\begin{align*}
    I_3=\sum_{k_2>-2.99l} 2^{12l-2\lambda l} \int_{2^{-0.01l}}^{2^{0.05l}}\int_{\left|\eta^{\prime\prime}\right|\sim 1}e^{i\left(s^{\prime\prime}\cdot 2^{(12-4\lambda)l}\right)\widetilde{\Phi_\lambda}(\xi^{\prime\prime},\eta^{\prime\prime})}\widehat{f_4}(2^{-9l}\xi^{\prime\prime}-2^{-3l}\eta^{\prime\prime},s^{\prime\prime})\widehat{f_5}(2^{-3l}\eta^{\prime\prime},s^{\prime\prime}) \,d\eta^{\prime\prime} ds^{\prime\prime}.
\end{align*}
If $\lambda l\ge 3D$, then Proposition 3.2 tells us that
\begin{align*}
    \begin{cases}
        &\left|\nabla_{\eta^{\prime\prime}}\widehat{f_4}(2^{-9l}\xi^{\prime\prime}-2^{-\lambda l}\eta^{\prime\prime},s^{\prime\prime})\right|\sim 2^{\lambda l}\cdot 2^{-\lambda l}\cdot\left|\widehat{f_4}(2^{-9l}\xi^{\prime\prime}-2^{-\lambda l}\eta^{\prime\prime},s^{\prime\prime})\right|\sim\left|\widehat{f_4}(2^{-9l}\xi^{\prime\prime}-2^{-\lambda l}\eta^{\prime\prime},s^{\prime\prime})\right| \\[8pt] &\left|\nabla_{\eta^{\prime\prime}}\widehat{f_5}(2^{-\lambda l}\eta^{\prime\prime},s^{\prime\prime})\right|\sim 2^{\lambda l}\cdot 2^{-\lambda l}\cdot\left|\widehat{f_5}(2^{-\lambda l}\eta^{\prime\prime},s^{\prime\prime})\right|\sim\left|\widehat{f_5}(2^{-\lambda l}\eta^{\prime\prime},s^{\prime\prime})\right| 
    \end{cases},
\end{align*}
since $\left|2^{-\lambda l}\eta^{\prime\prime}\right|, \left|2^{-9l}\xi^{\prime\prime}-2^{-\lambda l}\eta^{\prime\prime}\right|\sim 2^{-\lambda l}$
On the other hand, if $\lambda l\le 3D$, then we apply Proposition 3.3 to get that 
\begin{align*}
    \begin{cases}
        &\left|\nabla_{\eta^{\prime\prime}}\widehat{f_4}(2^{-9l}\xi^{\prime\prime}-2^{-\lambda l}\eta^{\prime\prime},s^{\prime\prime})\right|\lesssim \left|\widehat{f_4}(2^{-9l}\xi^{\prime\prime}-2^{-\lambda l}\eta^{\prime\prime},s^{\prime\prime})\right| \\[8pt] &\left|\nabla_{\eta^{\prime\prime}}\widehat{f_5}(2^{-\lambda l}\eta^{\prime\prime},s^{\prime\prime})\right|\lesssim \left|\widehat{f_5}(2^{-\lambda l}\eta^{\prime\prime},s^{\prime\prime})\right| 
    \end{cases}.
\end{align*}
Also note that $s^{\prime\prime}\cdot 2^{(12-4\lambda)l}\ge 2^{0.03l}$. 
Lemma 2.1 (substitute $l$ by $3l$) tells us that $\nabla_{\eta^{\prime\prime}}\widetilde{\Phi_\lambda}(\xi^{\prime\prime},\eta^{\prime\prime})\neq 0$, which implies that
\begin{align*}
    I_3&=\sum_{k_2>-2.99l} 2^{12l-2\lambda l} \int_{2^{-0.01l}}^{2^{0.05l}} O\left(\left(s^{\prime\prime}\cdot 2^{(12-4\lambda)l}\right)^{-N}\right) ds^{\prime\prime}
\end{align*}
So, by take $N,D$ large enough, we can get 
\begin{align*}
    \left|I_3\right|\le\sum_{k_2>-2.99l} A_N 2^{12l-2\lambda l-0.04Nl+0.05l}=\frac{4}{3} A_N 2^{(12.05-0.04N)l}\le 2^{-10l}
\end{align*}
On the other hand, as for $I_1$, let $\eta^{\prime\prime}=2^{3l} \eta$ and $s^{\prime\prime}=2^{-12l}s$ and we get
\begin{align*}
    I_1&=2^{6l} \int_{2^{-0.01l}}^{2^{0.05l}} \int_{\left|\eta^{\prime\prime}\right|\le 2^{-0.09l}}e^{is^{\prime\prime}\widetilde{\Phi}(\xi^{\prime\prime},\eta^{\prime\prime})}\widehat{f_4}(2^{-9l}\xi^{\prime\prime}-2^{-3l}\eta^{\prime\prime},s^{\prime\prime})\widehat{f_5}(2^{-3l}\eta^{\prime\prime},s^{\prime\prime})\,d\eta^{\prime\prime} ds^{\prime\prime}\\
    &=\sum_{-6l\le \widetilde{k_2}\le -0.09l}2^{6l} \int_{2^{-0.01l}}^{2^{0.05l}} \int_{\left|\eta^{\prime\prime}\right|\sim 2^{\widetilde{k_2}}}e^{is^{\prime\prime}\widetilde{\Phi}(\xi^{\prime\prime},\eta^{\prime\prime})}\widehat{f_4}(2^{-9l}\xi^{\prime\prime}-2^{-3l}\eta^{\prime\prime},s^{\prime\prime})\widehat{f_5}(2^{-3l}\eta^{\prime\prime},s^{\prime\prime})\,d\eta^{\prime\prime} ds^{\prime\prime}\\
    &\ \ \ +\sum_{\widetilde{k_2}\le -6l}2^{6l} \int_{2^{-0.01l}}^{2^{0.05l}} \int_{\left|\eta^{\prime\prime}\right|\sim 2^{\widetilde{k_2}}}e^{is^{\prime\prime}\widetilde{\Phi}(\xi^{\prime\prime},\eta^{\prime\prime})}\widehat{f_4}(2^{-9l}\xi^{\prime\prime}-2^{-3l}\eta^{\prime\prime},s^{\prime\prime})\widehat{f_5}(2^{-3l}\eta^{\prime\prime},s^{\prime\prime})\,d\eta^{\prime\prime} ds^{\prime\prime}
\end{align*}
So, by volume estimate and Proposition 3.2, we get 
\begin{align*}
    \left|I_1\right|&\le\sum_{-6l\le \widetilde{k_2}\le -0.09l} 2^{6l}\cdot 2^{0.05l}\cdot \pi\cdot 2^{2k_2}\cdot\left\|\widehat{f_4}\right\|_{L^\infty}\cdot\left\|\widehat{f_5}\right\|_{L^\infty}+\sum_{\widetilde{k_2}\le -6l}2^{6l}\cdot 2^{0.05l}\cdot \pi\cdot 2^{2k_2}\cdot\left\|\widehat{f_4}\right\|_{L^\infty}\cdot\left\|\widehat{f_5}\right\|_{L^\infty}\\
    &\le \sum_{-6l\le \widetilde{k_2}\le -0.09l} 2^{6l}\cdot 2^{0.05l}\cdot \pi\cdot 2^{2k_2}\cdot 2^{2l-\frac{2}{3}k_2}\cdot 2^{2l-\frac{2}{3}k_2}+\sum_{\widetilde{k_2}\le -6l}2^{6l}\cdot 2^{0.05l}\cdot \pi\cdot 2^{2k_2}\cdot 2^{6l}\cdot 2^{2l-\frac{2}{3}k_2}\\
    &\le 2^{10.05l}\cdot \sum_{-6l\le \widetilde{k_2}\le -0.09l} 2^{\frac{2}{3}k_2}+2^{14.05l}\cdot \sum_{\widetilde{k_2}\le -6l} 2^{\frac{4}{3}k_2}\lesssim 2^{9.99l}.
\end{align*}
Thus, in view of (\ref{3.14}), we have
\begin{align*}
    \reallywidehat{\mathcal{B}}_{\left[11.99l,12.05l\right],\left[-3.09l,-2.99l\right]}=C_\varphi \int_{2^{11.99l}}^{2^{12.05l}} \int_{\mathbb{R}^2} e^{is\Phi(\xi,\eta)}\widehat{f_4}(\xi-\eta,s)\widehat{f_5}(\eta,s)\,d\eta ds+O\left(2^{9.99l}\right)
\end{align*}
Let $\eta^{\prime\prime}=2^{3l} \eta$ and $s^{\prime\prime}=2^{-12l} s$. Then, we get
\begin{align*}
    &\reallywidehat{\mathcal{B}}_{\left[11.99l,12.05l\right],\left[-3.09l,-2.99l\right]}\\
    =&C_\varphi \cdot 2^{6l}\cdot\int_{2^{-0.01l}}^{2^{0.05l}} e^{is^{\prime\prime}\widetilde{\Phi}(\xi^{\prime\prime},\eta^{\prime\prime})}\widehat{f_4}(2^{-9l}\xi^{\prime\prime}-2^{-3l}\eta^{\prime\prime},s^{\prime\prime})\widehat{f_5}(2^{-3l}\eta^{\prime\prime},s^{\prime\prime})\,d\eta^{\prime\prime} ds^{\prime\prime}+O\left(2^{9.99l}\right) \\
    \triangleq& C_\varphi \cdot 2^{6l}\cdot\int_{2^{-0.01l}}^{2^{0.05l}}J(s^{\prime\prime})\,ds^{\prime\prime}+O\left(2^{9.99l}\right).\tag{3.15}\label{3.15}
\end{align*}
Now, if $\left|s^{\prime\prime}\right|\le 2^{-0.01l}$, then we decompose
\begin{align*}
    J(s^{\prime\prime})&=\int_{\mathbb{R}^2} e^{is^{\prime\prime} \Tilde{\Phi}(\xi^{\prime\prime},\eta^{\prime\prime})}\widehat{f_4}(2^{-9l}\xi^{\prime\prime}-2^{-3l}\eta^{\prime\prime},s^{\prime\prime})\widehat{f_5}(2^{-3l}\eta^{\prime\prime},s^{\prime\prime})\,d\eta^{\prime\prime}\\
    &=\int_{\left|\eta^{\prime\prime}\right|\le 2^l\cdot\left(s^{\prime\prime}\right)^{-\frac{1}{3}-0.01}} e^{is^{\prime\prime} \Tilde{\Phi}(\xi^{\prime\prime},\eta^{\prime\prime})}\widehat{f_4}(2^{-9l}\xi^{\prime\prime}-2^{-3l}\eta^{\prime\prime},s^{\prime\prime})\widehat{f_5}(2^{-3l}\eta^{\prime\prime},s^{\prime\prime})\,d\eta^{\prime\prime}\\
    &\ \ \ +\int_{\left|\eta^{\prime\prime}\right|\ge 2^l\cdot \left(s^{\prime\prime}\right)^{-\frac{1}{3}-0.01}} e^{is^{\prime\prime} \Tilde{\Phi}(\xi^{\prime\prime},\eta^{\prime\prime})}\widehat{f_4}(2^{-9l}\xi^{\prime\prime}-2^{-3l}\eta^{\prime\prime},s^{\prime\prime})\widehat{f_5}(2^{-3l}\eta^{\prime\prime},s^{\prime\prime})\,d\eta^{\prime\prime}\\
    &\triangleq I_{11}+I_{12}.
\end{align*}
In terms of $I_{11}$, since $\left|s^{\prime\prime}\right|\le 2^{-0.01l}$, we see that $2^{\widetilde{k_2}}\le 2^l\cdot (s^{\prime\prime})^{-\frac{1}{3}-0.01}\le 2^l$. 
Therefore, we can further decompose
\begin{align*}
    I_{11}&=\int_{2^{-3D}\le \left|\eta^{\prime\prime}\right|\le 2^l\cdot (s^{\prime\prime})^{-\frac{1}{3}-0.01}}e^{is^{\prime\prime}\widetilde{\Phi}(\xi^{\prime\prime},\eta^{\prime\prime})}\widehat{f_4}(2^{-9l}\xi^{\prime\prime}-2^{-3l}\eta^{\prime\prime},s^{\prime\prime})\widehat{f_5}(2^{-3l}\eta^{\prime\prime},s^{\prime\prime})\,d\eta^{\prime\prime} \\
    &\ \ \ +\sum_{-6l\le \widetilde{k_2}\le -3D}  \int_{\left|\eta^{\prime\prime}\right|\sim 2^{\widetilde{k_2}}}e^{is^{\prime\prime}\widetilde{\Phi}(\xi^{\prime\prime},\eta^{\prime\prime})}\widehat{f_4}(2^{-9l}\xi^{\prime\prime}-2^{-3l}\eta^{\prime\prime},s^{\prime\prime})\widehat{f_5}(2^{-3l}\eta^{\prime\prime},s^{\prime\prime})\,d\eta^{\prime\prime} \\
    &\ \ \ +\sum_{\widetilde{k_2}\le -6l}\int_{\left|\eta^{\prime\prime}\right|\sim 2^{\widetilde{k_2}}}e^{is^{\prime\prime}\widetilde{\Phi}(\xi^{\prime\prime},\eta^{\prime\prime})}\widehat{f_4}(2^{-9l}\xi^{\prime\prime}-2^{-3l}\eta^{\prime\prime},s^{\prime\prime})\widehat{f_5}(2^{-3l}\eta^{\prime\prime},s^{\prime\prime})\,d\eta^{\prime\prime} 
\end{align*}
By Proposition 3.1 and volume estimate, we see that
\begin{align*}
    \left|I_{11}\right|&\lesssim \left(2^{2l}\cdot (s^{\prime\prime})^{-\frac{2}{3}-0.02}\right)+\sum_{-6l\le \widetilde{k_2}\le -3D}  2^{2k_2}\cdot\left\|\widehat{f_4}\right\|_{L^\infty}\cdot\left\|\widehat{f_5}\right\|_{L^\infty}+\sum_{\widetilde{k_2}\le -6l} 2^{2k_2}\cdot\left\|\widehat{f_4}\right\|_{L^\infty}\cdot\left\|\widehat{f_5}\right\|_{L^\infty}\\
    &\lesssim \sum_{-6l\le \widetilde{k_2}\le -3D} 2^{2k_2}\cdot 2^{2l-\frac{2}{3}k_2}\cdot 2^{2l-\frac{2}{3}k_2}+\sum_{\widetilde{k_2}\le -6l}2^{2k_2}\cdot 2^{6l}\cdot 2^{2l-\frac{2}{3}k_2}\\
    &\lesssim \left(2^{2l}\cdot (s^{\prime\prime})^{-\frac{2}{3}-0.02}\right)+ 2^{4l}\cdot \sum_{-6l\le \widetilde{k_2}\le -3D} 2^{\frac{2}{3}k_2}+2^{8l}\cdot \sum_{\widetilde{k_2}\le -6l} 2^{\frac{4}{3}k_2}\lesssim 2^{4l}\cdot (s^{\prime\prime})^{-\frac{2}{3}-0.02}.\tag{3.16}\label{3.16}
\end{align*}
In terms of $I_{12}$, we let $s^{\prime\prime}\sim 2^{-\lambda l}$ ($\lambda\ge 0.01$) and write 
$$I_{12}=\sum_{\frac{\lambda}{3}+0.01\lambda+1<\Tilde{\lambda}\le 3.001 \atop \Tilde{\lambda}l\in\mathbb{Z}} \int_{\left|\eta^{\prime\prime}\right|\sim 2^{\Tilde{\lambda}l}} e^{is^{\prime\prime} \Tilde{\Phi}(\xi^{\prime\prime},\eta^{\prime\prime})}\widehat{f_4}(2^{-9l}\xi^{\prime\prime}-2^{-3l}\eta^{\prime\prime},s^{\prime\prime})\widehat{f_5}(2^{-3l}\eta^{\prime\prime},s^{\prime\prime})\,d\eta^{\prime\prime}.$$
Then, note that $\nabla_{\eta^{\prime\prime}}\left[s^{\prime\prime} \Tilde{\Phi}(\xi^{\prime\prime},\eta^{\prime\prime})\right]=s^{\prime\prime}\left(\frac{1}{2}\xi^{\prime\prime}+O(\left|\eta^{\prime\prime}\right|^2\eta^{\prime\prime})\right)$, which implies that $\left|\nabla_{\eta^{\prime\prime}}\left[s^{\prime\prime} \Tilde{\Phi}(\xi^{\prime\prime},\eta^{\prime\prime})\right]\right|\ge 2^{-\lambda l}\cdot 2^{3\Tilde{\lambda}l}\ge 2^{0.03\lambda l}\cdot 2^{3l}$. Also note that if $\left|2^{-9l}\xi^{\prime\prime}-2^{-3l}\eta^{\prime\prime}\right|\sim\left|2^{-3l}\eta^{\prime\prime}\right|\sim 2^{-3l+\widetilde{\lambda}l}\le 2^{-3D}$, then 
$$\left|\nabla_{\eta^{\prime\prime}}\widehat{f_4}(2^{-9l}\xi^{\prime\prime}-2^{-3l}\eta^{\prime\prime},s^{\prime\prime})\right|\sim 2^{-3l}\cdot 2^{(3-\widetilde{\lambda})l}\cdot \left|\widehat{f_4}(2^{-9l}\xi^{\prime\prime}-2^{-3l}\eta^{\prime\prime},s^{\prime\prime})\right|\sim 2^{-\widetilde{\lambda}l}\cdot\left|\widehat{f_4}(2^{-9l}\xi^{\prime\prime}-2^{-3l}\eta^{\prime\prime},s^{\prime\prime})\right|;$$
otherwise if $\left|2^{-3l}\eta^{\prime\prime}\right|\ge 2^{-3D}$, then
$$\left|\nabla_{\eta^{\prime\prime}}\widehat{f_4}(2^{-9l}\xi^{\prime\prime}-2^{-3l}\eta^{\prime\prime},s^{\prime\prime})\right|\lesssim 2^{-3l}\cdot \left|\widehat{f_4}(2^{-9l}\xi^{\prime\prime}-2^{-3l}\eta^{\prime\prime},s^{\prime\prime})\right|.$$
The same results hold for $\left|\nabla_{\eta^{\prime\prime}}\widehat{f_5}(2^{-3l}\eta^{\prime\prime},s^{\prime\prime})\right|$. Thus, we can keep doing integration by parts in $\eta^\prime$ such that for all $\frac{\lambda}{3}+0.01\lambda+1<\Tilde{\lambda}\le 3.001$, we have
$$\left|\int_{\left|\eta^{\prime\prime}\right|\sim 2^{\Tilde{\lambda}l}} e^{is^{\prime\prime} \Tilde{\Phi}(\xi^{\prime\prime},\eta^{\prime\prime})}\widehat{f_4}(2^{-9l}\xi^{\prime\prime}-2^{-3l}\eta^{\prime\prime},s^{\prime\prime})\widehat{f_5}(2^{-3l}\eta^{\prime\prime},s^{\prime\prime})\,d\eta^{\prime\prime}\right|\lesssim 2^{-11\lambda l}.$$
This gives
\begin{align*}
    \left|I_{12}\right|\lesssim \sum_{\frac{\lambda}{3}+0.01\lambda+1<\Tilde{\lambda}\le 3.001 \atop \Tilde{\lambda}l\in\mathbb{Z}} 2^{-11\lambda l} \lesssim \left|l\right|\cdot 2^{-11\lambda l}\lesssim 2^{-10\lambda l}\lesssim \left(s^{\prime\prime}\right)^{10}\lesssim \left(s^{\prime\prime}\right)^{-\frac{2}{3}-0.02}.\tag{3.17}\label{3.17}
\end{align*}
Combining (\ref{3.3}) and (\ref{3.4}), we conclude
\begin{align*}
    \left|\int_0^{2^{-0.01l}} J(s^{\prime\prime})\,ds^{\prime\prime}\right|\lesssim 2^{4l} \cdot\int_0^{2^{-0.01l}} (s^{\prime\prime})^{-\frac{2}{3}-0.02}\,ds^{\prime\prime}\lesssim 2^{4l}\cdot 2^{-\frac{47}{15000}l}\lesssim 2^{3.997l}.
\end{align*}
On the other hand, we suppose that $\left|s^{\prime\prime}\right|\ge 2^{0.05l}$. Note that $\widehat{f_4}(\xi)$ and $\widehat{f_5}(\xi)$ support on $\left\{\xi:\left|\xi\right|\le 4\right\}$. By Corollary 2.3, we know that there exists $\eta^{\prime\prime}(\xi^{\prime\prime})$ such that $\nabla_{\eta^{\prime\prime}}\Tilde{\Phi}(\xi^{\prime\prime},\eta^{\prime\prime}(\xi^{\prime\prime}))=0$ and 
$$\left|2^{-9l}\xi^{\prime\prime}-2^{-3l}\eta^{\prime\prime}(\xi^{\prime\prime})\right|\sim\left|2^{-3l}\eta^{\prime\prime}(\xi^{\prime\prime})\right|\sim 2^{-3l},$$
then we can write
\begin{align*}
    J(s^{\prime\prime})=C e^{is^{\prime\prime} \Tilde{\Phi}(\xi^{\prime\prime},\eta^{\prime\prime}(\xi^{\prime\prime}))}\widehat{f_4}(2^{-9l}\xi^{\prime\prime}-2^{-3l}\eta^{\prime\prime}(\xi^{\prime\prime}),s^{\prime\prime})\widehat{f_5}(2^{-3l}\eta^{\prime\prime}(\xi^{\prime\prime}),s^{\prime\prime})\cdot\frac{1}{s^{\prime\prime}}+O\left(\frac{1}{(s^{\prime\prime})^2}\right).
\end{align*}
Note that proposition 3.5 gives us that
\begin{align*}
    \begin{cases}
    \left|\partial_t \widehat{P_{-3l} f_4}(2^{-9l}\xi^{\prime\prime}-2^{-3l}\eta^{\prime\prime}(\xi^{\prime\prime})),t)\right|\approx \displaystyle{2^{2l}\cdot\frac{e^{i(\alpha_1\cdot 2^{-4l}\cdot t)}}{t}} \\
    \left|\partial_t \widehat{P_{-3l} f_5}(2^{-3l}\eta^{\prime\prime}(\xi^{\prime\prime})),t)\right|\approx \displaystyle{2^{2l}\cdot\frac{e^{i(\alpha_2\cdot 2^{-4l}\cdot t)}}{t}}
\end{cases},
\end{align*}
where we also see that $\alpha_1\approx \alpha_2 \approx \Tilde{\Phi}(\xi^{\prime\prime},\eta^{\prime\prime}(\xi^{\prime\prime}))$ due to the fact that $\left|\xi^{\prime\prime}\right|\sim 1$. Therefore, we have $\Tilde{\Phi}(\xi^{\prime\prime},\eta^{\prime\prime}(\xi^{\prime\prime}))-\alpha_i \cdot 2^{-4l}\triangleq \beta_i\neq 0$, where $i=1,2$. Now, integrate by parts in $s^{\prime\prime}$ and we get
\begin{align*}
    &\left|\int_{2^{0.05l}}^{+\infty} J(s^{\prime\prime})\,ds^{\prime\prime}\right|\\
    \lesssim& \left|\int_{2^{0.05l}}^{+\infty} \frac{e^{is^{\prime\prime} \Tilde{\Phi}(\xi^{\prime\prime},\eta^{\prime\prime}(\xi^{\prime\prime}))}}{s^{\prime\prime}}\widehat{f_4}(2^{-9l}\xi^{\prime\prime}-2^{-3l}\eta^{\prime\prime}(\xi^{\prime\prime}),s^{\prime\prime})\widehat{f_5}(2^{-3l}\eta^{\prime\prime}(\xi^{\prime\prime}),s^{\prime\prime})\,ds^{\prime\prime}\right|+\int_{2^{0.05l}}^{+\infty} \frac{ds^{\prime\prime}}{(s^{\prime\prime})^2}\\
    \lesssim& \left[\left|\int_{2^{0.05l}}^{+\infty} \frac{e^{is^{\prime\prime} \Tilde{\Phi}(\xi^{\prime\prime},\eta^{\prime\prime}(\xi^{\prime\prime}))}}{ (s^{\prime\prime})^2}\widehat{f_4}(2^{-9l}\xi^{\prime\prime}-2^{-3l}\eta^{\prime\prime}(\xi^{\prime\prime}),s^{\prime\prime})\widehat{f_5}(2^{-3l}\eta^{\prime\prime}(\xi^{\prime\prime}),s^{\prime\prime})\,ds^{\prime\prime}\right|\right.\\
    &\left.+\left|\int_{2^{0.05l}}^{+\infty} \frac{e^{i\beta_2 s^{\prime\prime} }}{ (s^{\prime\prime})^2}\cdot\widehat{f_4}(2^{-9l}\xi^{\prime\prime}-2^{-3l}\eta^{\prime\prime}(\xi^{\prime\prime}),s^{\prime\prime})\,ds^{\prime\prime}\right|+\left|\int_{2^{0.05l}}^{+\infty} \frac{e^{i\beta_1 s^{\prime\prime} }}{ (s^{\prime\prime})^2}\cdot\widehat{f_5}(2^{-3l}\eta^{\prime\prime}(\xi^{\prime\prime}),s^{\prime\prime})\,ds^{\prime\prime}\right|\right]\\
    &+\frac{e^{is^{\prime\prime} \Tilde{\Phi}(\xi^{\prime\prime},\eta^{\prime\prime}(\xi^{\prime\prime}),s^{\prime\prime})}}{2^{0.05l}}+2^{-0.05l}\\
    \lesssim&\,2^{4l}\cdot 2^{-0.05l}\approx 2^{3.95l}.\tag{3.18}\label{3.18}
\end{align*}
Finally, in view of (\ref{3.6a}), we see that
\begin{align*}
    \reallywidehat{\mathcal{B}}_{\left[11.99l,12.05l\right],\left[-3.09l,-2.99l\right]}=C_\varphi \cdot 2^{6l}\cdot\int_0^{+\infty}J(s^{\prime\prime})\,ds^{\prime\prime}+O\left(2^{9.997l}\right).\tag{3.19}\label{3.19}
\end{align*}
Recall that Proposition 3.1 also tells us that if $t\ge 2^{4.05l}$ and $\left|\xi\right|\sim 2^{-3l}$, then 
$$\widehat{P_{-3l}f_i}(\xi,t)\approx 2^{2l}+e(\xi,t),$$
where $i=4,5$ and $\left|e(\xi,t)\right|\lesssim 2^{1.99l}$. Also recall that proposition 3.5 tells us that
\begin{align*}
    \begin{cases}
    \left|\partial_t e(2^{-9l}\xi^{\prime\prime}-2^{-3l}\eta^{\prime\prime}(\xi^{\prime\prime})),t)\right|\approx \displaystyle{2^{2l}\cdot\frac{e^{i(\alpha_3\cdot 2^{-4l}\cdot t)}}{t}} \\
    \left|\partial_t e(2^{-3l}\eta^{\prime\prime}(\xi^{\prime\prime})),t)\right|\approx \displaystyle{2^{2l}\cdot\frac{e^{i(\alpha_4\cdot 2^{-4l}\cdot t)}}{t}}
\end{cases},
\end{align*}
where we also see that $\alpha_3\approx \alpha_4 \approx \Tilde{\Phi}(\xi^{\prime\prime},\eta^{\prime\prime}(\xi^{\prime\prime}))$ due to the fact that $\left|\xi^{\prime\prime}\right|\sim 1$. Therefore, we can compute
\begin{align*}
    \int_{2^{-0.005l}}^{+\infty}J(s^{\prime\prime})\,ds^{\prime\prime}&=\int_{2^{-0.005l}}^{+\infty}\int_{\mathbb{R}^2} e^{is^{\prime\prime}\widetilde{\Phi}(\xi^{\prime\prime},\eta^{\prime\prime})}\widehat{f_4}(2^{-9l}\xi^{\prime\prime}-2^{-3l}\eta^{\prime\prime},s^{\prime\prime})\widehat{f_5}(2^{-3l}\eta^{\prime\prime},s^{\prime\prime})\,d\eta^{\prime\prime} ds^{\prime\prime}\\
    &\approx 2^{4l}\cdot\int_{2^{-0.005l}}^{+\infty}\int_{\mathbb{R}^2}e^{is^{\prime\prime}\widetilde{\Phi}(\xi^{\prime\prime},\eta^{\prime\prime})}\,d\eta^{\prime\prime}ds^{\prime\prime}\\
    &\ \ \ \ +2^{2l}\cdot\int_{2^{-0.005l}}^{+\infty}\int_{\mathbb{R}^2}e^{is^{\prime\prime}\widetilde{\Phi}(\xi^{\prime\prime},\eta^{\prime\prime})}e(2^{-3l}\eta^{\prime\prime},s^{\prime\prime})\,d\eta^{\prime\prime}ds^{\prime\prime}\\
    &\ \ \ \ +2^{2l}\cdot\int_{2^{-0.005l}}^{+\infty}\int_{\mathbb{R}^2}e^{is^{\prime\prime}\widetilde{\Phi}(\xi^{\prime\prime},\eta^{\prime\prime})}e(2^{-9l}\xi^{\prime\prime}-2^{-3l}\eta^{\prime\prime},s^{\prime\prime})\,d\eta^{\prime\prime}ds^{\prime\prime}\\
    &\ \ \ \ +\int_{2^{-0.005l}}^{+\infty}\int_{\mathbb{R}^2}e^{is^{\prime\prime}\widetilde{\Phi}(\xi^{\prime\prime},\eta^{\prime\prime})}e(2^{-3l}\eta^{\prime\prime},s^{\prime\prime})e(2^{-9l}\xi^{\prime\prime}-2^{-3l}\eta^{\prime\prime},s^{\prime\prime})\,d\eta^{\prime\prime}ds^{\prime\prime}\\
    &\triangleq J_1+J_2+J_3+J_4.
\end{align*}
As for $J_4$, we define $\Tilde{\Phi}(\xi^{\prime\prime},\eta^{\prime\prime}(\xi^{\prime\prime}))-\alpha_i \cdot 2^{-4l}\triangleq \beta_i\neq 0$ ($i=3,4$) like before and again let $\eta^{\prime\prime}(\xi^{\prime\prime})$ be such that $\nabla_{\eta^{\prime\prime}}\Tilde{\Phi}(\xi^{\prime\prime},\eta^{\prime\prime}(\xi^{\prime\prime}))=0$ by Corollary 2.3. Now, integrate by parts in $s^{\prime\prime}$ and we get
\begin{align*}
    &\left|\int_{2^{-0.005l}}^{+\infty}\int_{\mathbb{R}^2}e^{is^{\prime\prime}\widetilde{\Phi}(\xi^{\prime\prime},\eta^{\prime\prime})}e(2^{-9l}\xi^{\prime\prime}-2^{-3l}\eta^{\prime\prime},s^{\prime\prime}) e(2^{-3l}\eta^{\prime\prime},s^{\prime\prime})\,d\eta^{\prime\prime}ds^{\prime\prime}\right|\\
    \lesssim& \left|\int_{2^{-0.005l}}^{+\infty} \frac{e^{is^{\prime\prime} \Tilde{\Phi}(\xi^{\prime\prime},\eta^{\prime\prime}(\xi^{\prime\prime}))}}{s^{\prime\prime}}e(2^{-9l}\xi^{\prime\prime}-2^{-3l}\eta^{\prime\prime}(\xi^{\prime\prime}),s^{\prime\prime})e(2^{-3l}\eta^{\prime\prime}(\xi^{\prime\prime}),s^{\prime\prime})\,ds^{\prime\prime}\right|+\int_{2^{-0.005l}}^{+\infty} \frac{ds^{\prime\prime}}{(s^{\prime\prime})^2}\\
    \lesssim& \left[\left|\int_{2^{-0.005l}}^{+\infty} \frac{e^{is^{\prime\prime} \Tilde{\Phi}(\xi^{\prime\prime},\eta^{\prime\prime}(\xi^{\prime\prime}))}}{ (s^{\prime\prime})^2}e(2^{-9l}\xi^{\prime\prime}-2^{-3l}\eta^{\prime\prime}(\xi^{\prime\prime}),s^{\prime\prime})e(2^{-3l}\eta^{\prime\prime}(\xi^{\prime\prime}),s^{\prime\prime})\,ds^{\prime\prime}\right|\right.\\
    &\left.+\left|\int_{2^{-0.005l}}^{+\infty} \frac{e^{i\beta_3 s^{\prime\prime} }}{ (s^{\prime\prime})^2}\cdot e(2^{-9l}\xi^{\prime\prime}-2^{-3l}\eta^{\prime\prime}(\xi^{\prime\prime}),s^{\prime\prime})\,ds^{\prime\prime}\right|+\left|\int_{2^{-0.005l}}^{+\infty} \frac{e^{i\beta_4 s^{\prime\prime} }}{ (s^{\prime\prime})^2}\cdot e(2^{-3l}\eta^{\prime\prime}(\xi^{\prime\prime}),s^{\prime\prime})\,ds^{\prime\prime}\right|\right]\\
    &+\frac{e^{i\cdot 2^{-0.005l}\cdot \Tilde{\Phi}(\xi^{\prime\prime},\eta^{\prime\prime}(\xi^{\prime\prime}))}}{2^{-0.005l}}+\cdot 2^{0.005l}\\
    \lesssim& \,2^{1.99l}\cdot 2^{1.99l}\cdot 2^{0.005l}\approx 2^{3.985l}.
\end{align*}
By a similar computation, we can conclude that $\left|J_2\right|,\left|J_3\right|\lesssim 2^{3.995l}$. Thus, we get
\begin{align*}
    \int_{2^{-0.005l}}^{+\infty}J(s^{\prime\prime})\,ds^{\prime\prime}\approx 2^{4l}\cdot\int_{2^{-0.005l}}^{+\infty}\int_{\mathbb{R}^2}e^{is^{\prime\prime}\widetilde{\Phi}(\xi^{\prime\prime},\eta^{\prime\prime})}\,d\eta^{\prime\prime}ds^{\prime\prime}+O\left(2^{3.995l}\right).
\end{align*}
On the other hand, use the method in (\ref{3.16}) and (\ref{3.17}) and we 
 can compute
\begin{align*}
    \int_0^{2^{-0.005l}}J(s^{\prime\prime})\,ds^{\prime\prime}&=\int_0^{2^{-0.005l}}\int_{\mathbb{R}^2} e^{is^{\prime\prime}\widetilde{\Phi}(\xi^{\prime\prime},\eta^{\prime\prime})}\widehat{f_4}(2^{-9l}\xi^{\prime\prime}-2^{-3l}\eta^{\prime\prime},s^{\prime\prime})\widehat{f_5}(2^{-3l}\eta^{\prime\prime},s^{\prime\prime})\,d\eta^{\prime\prime} ds^{\prime\prime}\\
    &\approx 2^{4l}\cdot\int_0^{2^{-0.005l}}\int_{\mathbb{R}^2}e^{is^{\prime\prime}\widetilde{\Phi}(\xi^{\prime\prime},\eta^{\prime\prime})}\,d\eta^{\prime\prime}ds^{\prime\prime}\\
    &\ \ \ \ +2^{2l}\cdot\int_0^{2^{-0.005l}}\int_{\mathbb{R}^2}e^{is^{\prime\prime}\widetilde{\Phi}(\xi^{\prime\prime},\eta^{\prime\prime})}e(2^{-3l}\eta^{\prime\prime},s^{\prime\prime})\,d\eta^{\prime\prime}ds^{\prime\prime}\\
    &\ \ \ \ +2^{2l}\cdot\int_0^{2^{-0.005l}}\int_{\mathbb{R}^2}e^{is^{\prime\prime}\widetilde{\Phi}(\xi^{\prime\prime},\eta^{\prime\prime})}e(2^{-9l}\xi^{\prime\prime}-2^{-3l}\eta^{\prime\prime},s^{\prime\prime})\,d\eta^{\prime\prime}ds^{\prime\prime}\\
    &\ \ \ \ +\int_0^{2^{-0.005l}}\int_{\mathbb{R}^2}e^{is^{\prime\prime}\widetilde{\Phi}(\xi^{\prime\prime},\eta^{\prime\prime})}e(2^{-3l}\eta^{\prime\prime},s^{\prime\prime})e(2^{-9l}\xi^{\prime\prime}-2^{-3l}\eta^{\prime\prime},s^{\prime\prime})\,d\eta^{\prime\prime}ds^{\prime\prime}\\
    &\approx 2^{4l}\cdot\int_0^{2^{-0.005l}}\int_{\mathbb{R}^2}e^{is^{\prime\prime}\widetilde{\Phi}(\xi^{\prime\prime},\eta^{\prime\prime})}\,d\eta^{\prime\prime}ds^{\prime\prime}+O\left(2^{3.9785l}\right).
\end{align*}
Next, we see that
\begin{align*}
    \int_0^{+\infty}J(s^{\prime\prime})\,ds^{\prime\prime}&=\int_0^{2^{-0.005l}}J(s^{\prime\prime})\,ds^{\prime\prime}+\int_{2^{-0.005l}}^{+\infty}J(s^{\prime\prime})\,ds^{\prime\prime}\\
    &\approx 2^{4l}\cdot\int_0^{+\infty}\int_{\mathbb{R}^2}e^{is^{\prime\prime}\widetilde{\Phi}(\xi^{\prime\prime},\eta^{\prime\prime})}\,d\eta^{\prime\prime}ds^{\prime\prime}+O\left(2^{3.995l}\right).
\end{align*}
Note that
\begin{align*}
    \int_0^{+\infty}\int_{\mathbb{R}^2} e^{is^{\prime\prime}\widetilde{\Phi}(\xi^{\prime\prime},\eta^{\prime\prime})} \,d\eta^{\prime\prime} ds^{\prime\prime}\neq 0
\end{align*}
is a nonzero constant independent of $l$, which implies that
$$\int_0^{+\infty}J(s^{\prime\prime})\,ds^{\prime\prime}\approx 2^{4l}.$$
If $m\ge 0.05l$, then in view of (\ref{3.18}), we see that
$$\int_0^{2^m}J(s^{\prime\prime})\,ds^{\prime\prime}\approx\int_0^{+\infty}J(s^{\prime\prime})\,ds^{\prime\prime}\approx 2^{4l}.$$
Now, we compute
\begin{align*}
    \widehat{P_{-9l}f_6}(\xi,t)&=\int_0^{2^m}\int_{\mathbb{R}^2} e^{is\Phi(\xi,\eta)} \widehat{f_4}(\xi-\eta)\widehat{f_5}(\eta) \,d\eta ds \\
    &=2^{6l}\cdot\int_0^{2^{m-12l}}\int_{\mathbb{R}^2} e^{is^{\prime\prime}\widetilde{\Phi}(\xi^{\prime\prime},\eta^{\prime\prime})} \widehat{f_4}(2^{-3l}\xi^{\prime\prime}-2^{-l}\eta^{\prime\prime})\widehat{f_5}(2^{-l}\eta^{\prime\prime}) \,d\eta^{\prime\prime} ds^{\prime\prime}\\
    &\approx 2^{6l}\cdot \int_0^{+\infty} J(s^{\prime\prime})\,ds^{\prime\prime} \approx 2^{6l}\cdot 2^{4l}\approx 2^{10l}.
\end{align*}
Namely, $\left|\widehat{P_{-9l}f_6}(\xi,t)\right|\approx 2^{10l}$, provided that $m\ge 12.05l$ and $\left|\xi\right|\sim 2^{-9l}$. Moreover, we have that
\begin{align*}
   \widehat{P_{-9l}f_6}(\xi,t)=\reallywidehat{\mathcal{B}}_{\left[11.99l,12.05l\right],\left[-3.09l,-2.99l\right]}+O\left(2^{9.997l}\right).
\end{align*}
\end{proof}
\vspace{0.8em}
\begin{theorem}
    For any $l\ge 10D\gg 1$, we fix $l$. If $t\ge 2^{12.05l}$, then $\left\|P_{-9l} f_6\right\|_{L^2}\approx 2^l\gg 1$
\end{theorem}
\begin{proof}
Denote $\left|E_\eta\right|=\mbox{the volume of annulus }\left|\xi\right|\sim 2^{-9l}$. Using Proposition 3.7, this follows from
$$\inf_{\left|\xi\right|\sim 2^{-9l}}\left|\widehat{P_{-9l}f_6}(\xi,t)\right|\cdot\left|E_\eta\right|^{1/2}\lesssim\left\|\widehat{P_{-9l}f_6}\right\|_{L^2}\lesssim \sup_{\left|\xi\right|\sim 2^{-9l}}\left|\widehat{P_{-9l}f_6}(\xi,t)\right|\cdot\left|E_\eta\right|^{1/2}.$$
\end{proof}

\bibliographystyle{plain}
\bibliography{sample}

\end{document}